\documentclass{amsart}

\usepackage{mathrsfs}
\usepackage{amscd}
\usepackage{amsmath}
\usepackage{amssymb}
\usepackage{amsthm}
\usepackage{bm}
\usepackage{epsf}
\usepackage{latexsym}
\usepackage{verbatim}
\usepackage[all, cmtip]{xy}
\usepackage{tikz}
\usetikzlibrary{positioning}
\usetikzlibrary{matrix}
\usepackage{float}
\usepackage{hyperref}
\usepackage{comment}
\usepackage{ytableau}

\tikzstyle{bsq}=[rectangle, draw, thick, minimum width=.5cm, minimum height=.5cm]
\tikzstyle{bver}=[rectangle, draw, thick, minimum width=1cm, minimum height=2cm]
\tikzstyle{bhor}=[rectangle, draw, thick, minimum width=2cm, minimum height=1cm]

\usepackage[left=3.2cm, right=3.2cm]{geometry}

\newtheorem{theorem}{Theorem}[section]
\newtheorem{definition}[theorem]{Definition}
\newtheorem{lemma}[theorem]{Lemma}
\newtheorem{conjecture}[theorem]{Conjecture}
\newtheorem{corollary}[theorem]{Corollary}
\newtheorem{proposition}[theorem]{Proposition}

\newtheorem{varexample}[theorem]{Example}

\theoremstyle{definition}
\newtheorem{remark}[theorem]{Remark}
\newtheorem{algorithm}[theorem]{Algorithm}

\newcommand{\PP}{\mathbb{P}}

\newcommand{\ZZ}{\mathbb{Z}}
\newcommand{\TT}{\mathbb{T}}

\newcommand{\cO}{\mathcal{O}}

\newcommand{\bmu}{\bm{\mu}}
\newcommand{\blam}{\bm{\lambda}}
\newcommand{\bC}{\mathbf{C}}

\newcommand{\Trop}{\operatorname{Trop}}

\newcommand{\Jac}{\operatorname{Jac}}

\newcommand{\Pic}{\operatorname{Pic}}
\newcommand{\Sym}{\operatorname{Sym}}

\newcommand{\rk}{\mathrm{rk}}

\newenvironment{example}{\begin{varexample}
\begin{normalfont}}{\end{normalfont}
\end{varexample}}

\begin{document}
\title{Tropical Methods in Hurwitz-Brill-Noether Theory}

\author{Kaelin Cook-Powell}
\author{David Jensen}

\maketitle

\begin{abstract}
Splitting type loci are the natural generalizations of Brill-Noether varieties for curves with a distinguished map to the projective line.  We give a tropical proof of a theorem of H. Larson, showing that splitting type loci have the expected dimension for general elements of the Hurwitz space.  Our proof uses an explicit description of splitting type loci on a certain family of tropical curves.  We further show that these tropical splitting type loci are connected in codimension one, and describe an algorithm for computing their cardinality when they are zero-dimensional.  We provide a conjecture for the numerical class of splitting type loci, which we confirm in a number of cases.
\end{abstract}

\section{Introduction}
\label{Sec:Intro}

The Picard variety of a curve $C$ is stratified by the subschemes $W^r_d (C)$, parameterizing line bundles of degree $d$ and rank at least $r$.  The study of these subschemes, known as Brill-Noether theory, is a central area of research in algebraic geometry.  The celebrated Brill-Noether Theorem of Griffiths and Harris says that, if $C \in \mathcal{M}_g$ is general, then the varieties $W^r_d (C)$ are equidimensional of the expected dimension, with the convention that a variety of negative dimension is empty \cite{GH80}.

If $C$ is not general, what can we say about its Brill-Noether theory?  The \emph{gonality} of $C$ is the smallest integer $k$ such that $W^1_k (C)$ is nonempty, and a consequence of the Brill-Noether Theorem is that the gonality of a general curve is $\lfloor \frac{g+3}{2} \rfloor$.  If we assume that $C$ has smaller gonality than this, what effect does this assumption have on the dimensions of $W^r_d (C)$ for other values of $r$ and $d$?  Along these lines, several recent papers have focused on the Brill-Noether theory of curves that are general in the Hurwitz space $\mathcal{H}_{g,k}$, rather than the moduli space $\mathcal{M}_g$ \cite{CoppensMartens, CoppensMartens2, PfluegerkGonal, JensenRanganthan, Larson, CPJ, LU, CLRW}.  The Hurwitz space $\mathcal{H}_{g,k}$ parameterizes degree $k$ branched covers of $\PP^1$, where the source has genus $g$.  If $k < \lfloor \frac{g+3}{2} \rfloor$ and $(C, \pi) \in \mathcal{H}_{g,k}$ is general, then the varieties $W^r_d (C)$ can have multiple components of varying dimensions, prohibiting a naive generalization of the Brill-Noether Theorem.

In this setting, however, the Picard variety of $C$ admits a more refined stratification.  We say that a line bundle $L \in \Pic (C)$ has \emph{splitting type} $\bmu  = (\mu_1 , \ldots , \mu_k)$ if $\pi_* L \cong \oplus_{i=1}^k \cO (\mu_i )$.  (See Section~\ref{Sec:Split}.)  Since the splitting type of a line bundle determines that line bundle's rank and degree, it is a more refined invariant.  The \emph{splitting type locus} $W^{\bmu} (C) \subseteq \Pic (C)$ parameterizing line bundles of splitting type $\bmu$ is locally closed, of expected codimension
\[
\vert \bmu \vert := \sum_{i < j} \max \{ 0 , \mu_j - \mu_i - 1 \} .
\]
In \cite{Larson}, H. Larson proves an analogue of the Brill-Noether Theorem for the strata $W^{\bmu} (C)$.

\begin{theorem} \cite{Larson}
\label{Thm:MainThm}
Let $(C,\pi) \in \mathcal{H}_{g,k}$ be general.  If $g \geq \vert \bmu \vert$, then
\[
\mathrm{dim} W^{\bmu} (C) = g - \vert \bmu \vert .
\]
If $g < \vert \bmu \vert$, then $W^{\bmu} (C)$ is empty.
\end{theorem}

Theorem~\ref{Thm:MainThm} is proven by considering analogous closed strata $\overline{W}^{\bmu}(C)$ containing $W^{\bmu}(C)$.  We refer the reader to Section~\ref{Sec:Split} for a precise definition.  As in the original Brill-Noether Theorem, the fact that the dimension of $\overline{W}^{\bmu} (C)$ is at least $g - \vert \bmu \vert$ holds for all $(C,\pi) \in \mathcal{H}_{g,k}$.  This follows from standard results about degeneracy loci, provided that $\overline{W}^{\bmu} (C)$ is nonempty.  Larson demonstrates the nonemptiness of $\overline{W}^{\bmu} (C)$ by showing that a certain intersection number is nonzero.

The fact that the dimension of $\overline{W}^{\bmu} (C)$ is at most $g - \vert \bmu \vert$ is much deeper.  In this paper, we give a new proof of this result using tropical and combinatorial techniques.  In Section~\ref{Sec:COL}, we define tropical analogues of splitting type loci, and prove that these contain the image of the corresponding splitting type loci under tropicalization.  Our approach builds on earlier work exploring the divisor theory of a certain family of tropical curves known as \emph{chains of loops} \cite{CDPR, PfluegerkGonal, JensenRanganthan, CPJ}.  Theorem~\ref{Thm:MainThm} is a consequence of the following result.

\begin{theorem}
\label{thm:TropEquiDim}
Let $\Gamma$ be a $k$-gonal chain of loops of genus $g$.  If $g \geq \vert \bmu \vert$, then $\overline{W}^{\bmu} (\Gamma)$ is equidimensional and
\[
\mathrm{dim} \overline{W}^{\bmu}(\Gamma) = g - \vert \bmu \vert .
\]
If $g < \vert \bmu \vert$, then $\overline{W}^{\bmu} (\Gamma)$ is empty.
\end{theorem}

\subsection{Tropical Splitting Type Loci}

In her proof of Theorem~\ref{Thm:MainThm}, Larson uses the theory of limit linear series on a chain of elliptic curves.  Remarkably, however, her proof does not require a description of splitting type loci on this degenerate curve.  That is, it is not necessary for her to classify those limit linear series that are limits of line bundles with a given splitting type $\bmu$.  In contrast, our proof of Theorem~\ref{thm:TropEquiDim} follows from an explicit description of $\overline{W}^{\bmu} (\Gamma)$.  This description is used to prove new results and formulate Conjectures~\ref{Conj:Class} and~\ref{Conj:Surj}.

Our description of $\overline{W}^{\bmu} (\Gamma)$ builds on the earlier work of \cite{CDPR, PfluegerkGonal, JensenRanganthan, CPJ} mentioned above.  The main technical result of \cite{CDPR} is a classification of special divisor classes on chains of loops, when the lengths of the edges are sufficiently general.  Specifically, if $\Gamma$ is such a chain of loops, then $W^r_d (\Gamma)$ is union of tori $\TT (t)$, where each torus corresponds to a standard Young tableau $t$ on a certain rectangular partition.  This result was generalized in \cite{PfluegerCycles, PfluegerkGonal} to chains of loops with arbitrary edge lengths.  If $\Gamma$ is the $k$-gonal chain of loops referred to in Theorem~\ref{thm:TropEquiDim}, then $W^r_d (\Gamma)$ is again a union of tori $\TT (t)$ indexed by rectangular tableaux, but here the tableaux are non-standard.  Instead, the tableaux are required to satisfy an arithmetic condition known as \emph{$k$-uniform displacement} (see Definition~\ref{Def:Displacement}).

Given a splitting type $\bmu \in \ZZ^k$, we define a partition $\lambda (\bmu)$ in Definition~\ref{def:SplitTypePartition}.  We call a partition of this type a $k$-\emph{staircase}.  Our description of $\overline{W}^{\bmu}(\Gamma)$ is analogous to that of $W^r_d (\Gamma)$ mentioned above.

\begin{theorem}
\label{thm:TropClassify}
Let $\Gamma$ be a $k$-gonal chain of loops of genus $g$.  Then 
\[
\overline{W}^{\bmu}(\Gamma) = \bigcup \TT(t),
\]
where the union is over all $k$-uniform displacement tableau $t$ on $\lambda(\bmu)$ with alphabet $[g]$.
\end{theorem}

We prove Theorem~\ref{thm:TropClassify} in Section~\ref{Sec:TSL}.  The remainder of the paper uses this classification to establish various geometric properties of the tropical splitting type loci $\overline{W}^{\bmu} (\Gamma)$.  For example, we compute the dimension of $\overline{W}^{\bmu} (\Gamma)$ in Section~\ref{Sec:Dim}, proving Theorem~\ref{thm:TropEquiDim}.  In Section~\ref{Sec:Connect}, we study the connectedness of tropical splitting type loci.

\begin{theorem}
\label{thm:connect}
Let $\Gamma$ be a $k$-gonal chain of loops of genus $g$.  If $g > \vert \bmu \vert$, then $\overline{W}^{\bmu} (\Gamma)$ is connected in codimension 1.
\end{theorem}

Unfortunately, the connectedness of $\overline{W}^{\bmu} (\Gamma)$ does not imply that of $\overline{W}^{\bmu}(C)$ for a general $(C,\pi) \in \mathcal{H}_{g,k}$.  Theorem~\ref{thm:connect} is nevertheless interesting for at least two reasons.  First, by \cite[Theorem~1.2]{Larson}, we know that the locally closed stratum $W^{\bmu} (C)$ is smooth for a general $(C,\pi) \in \mathcal{H}_{g,k}$, so it is irreducible if and only if it is connected.  Theorem~\ref{thm:connect} therefore suggests that $W^{\bmu} (C)$ is irreducible if it is positive dimensional, as predicted in \cite[Conjecture~1.2]{CPJ}.  Second, by \cite[Theorem~1]{CartwrightPayne}, the tropicalization of a variety is equidimensional and connected in codimension one, so Theorems~\ref{thm:TropEquiDim} and~\ref{thm:connect} can be seen as evidence that $\overline{W}^{\bmu} (\Gamma)$ is the tropicalization of $\overline{W}^{\bmu} (C)$ (see Conjecture~\ref{Conj:Surj} below).  

\subsection{Numerical Classes}

These geometric results follow from a careful study of $k$-staircases and $k$-uniform displacement tableaux.  These combinatorial objects are explored in Section~\ref{Sec:Partitions}.  Staircases belong to a wider class of partitions, known as \emph{$k$-cores}, which have been studied extensively in other contexts.  (See, for example, \cite{LLMSSZ}.)  The set $\mathcal{P}_k$ of $k$-cores is a ranked poset (Corollary~\ref{cor:PosetIsPoset}), with cover relations given by \emph{upward displacements} in the sense of \cite[Definition~6.1]{PfluegerNegativeBN}.  We write $\mathcal{P}_k (\lambda)$ for the interval below $\lambda \in \mathcal{P}_k$.  In Section~\ref{Sec:Count}, we use these observations to compute the cardinality of zero-dimensional tropical splitting type loci.

\begin{theorem}
\label{thm:count}
Let $\Gamma$ be a $k$-gonal chain of loops of genus $g$.  If $g = \vert \bmu \vert$, then $\vert \overline{W}^{\bmu} (\Gamma) \vert$ is equal to the number of maximal chains in $\mathcal{P}_k (\lambda (\bmu))$.
\end{theorem}

The number of maximal chains in $\mathcal{P}_k (\lambda (\bmu))$ has received significant interest in the combinatorics and representation theory literature, and has connections to the affine symmetric group.  More precisely, there is a bijection between such maximal chains and reduced words in the affine symmetric group \cite{LapointeMorse}.  For this reason, several of our results have equivalent formulations in terms of these groups (see Remarks~\ref{Rmk:Dim} and~\ref{Rmk:Connect}).  There is currently no known closed form expression for these numbers, but they satisfy a simple recurrence (Lemma~\ref{lem:recur}) that allows one to compute a given number in polynomial time (Algorithm~\ref{Alg:Recur}).

Theorem~\ref{thm:count} has implications beyond the zero-dimensional case.  In \cite[Lemma~5.4]{Larson}, Larson shows that the numerical class of $\overline{W}^{\bmu} (C)$ in $\Pic (C)$ is of the form $a_{\bmu} \Theta^{\vert \bmu \vert}$, where the coefficient $a_{\bmu}$ is independent of the genus.  To compute the coefficient $a_{\bmu}$, therefore, it suffices to compute the cardinality of $\overline{W}^{\bmu} (C)$ in the case where $g = \vert \bmu \vert$.  In this way, Theorem~\ref{thm:count} suggests the following conjecture.

\begin{conjecture}
\label{Conj:Class}
Let $(C,\pi) \in \mathcal{H}_{g,k}$ be general.  The numerical class of $\overline{W}^{\bmu} (C)$ in $\Pic^{d(\bmu)} (C)$ is
\[
\left[ \overline{W}^{\bmu} (C) \right] = \frac{1}{\vert \bmu \vert !} \cdot \alpha (\mathcal{P}_k (\lambda (\bmu))) \cdot \Theta^{\vert \bmu \vert} ,
\]
where $\alpha (\mathcal{P})$ denotes the number of maximal chains in the poset $\mathcal{P}$.
\end{conjecture}

At the end of Section~\ref{Sec:Count}, we provide evidence for Conjecture~\ref{Conj:Class}, in the form of numerous examples where it holds.  We also compute the number of maximal chains in $\mathcal{P}_k (\lambda (\bmu))$ for some infinite families of splitting types where the class of $\overline{W}^{\bmu} (C)$ is unknown.  For such families, these numbers form well-known integer sequences, including binomial coefficients (Examples~\ref{Ex:Trigonal} and~\ref{Ex:OneRowOneCol}), geometric sequences (Example~\ref{Ex:Four}), Catalan numbers (Example~\ref{Ex:Catalan}), and Fibonacci numbers (Example~\ref{Ex:Five}).  Conjecture~\ref{Conj:Class} would be implied by the following.

\begin{conjecture}
\label{Conj:Surj}
Let $\Gamma$ be a $k$-gonal chain of loops, and let $C$ be a curve of genus $g$ and gonality $k$ over a nonarchimedean field $K$ with skeleton $\Gamma$.  Then the tropicalization map
\[
\Trop : \overline{W}^{\bmu} (C) \to \overline{W}^{\bmu} (\Gamma)
\]
is surjective.  Moreover, if $g = \vert \bmu \vert$, then it is a bijection.
\end{conjecture}

Conjecture~\ref{Conj:Surj} is known to hold in several important cases.  It is the main result of \cite{CartwrightJensenPayne} in the case where $\Gamma$ has generic edge lengths (or equivalently, when $k = \lfloor \frac{g+3}{2} \rfloor$).  The main results of \cite{CPJ} and \cite{JensenRanganthan} combined show that the tropicalization map is surjective for the ``maximal'' splitting types $\bmu_{\alpha}$ of \cite[Definition~2.5]{CPJ}.  We do not, however, know that it is bijective in the zero-dimensional case.  Conjecture~\ref{Conj:Surj} remains open in many cases where Conjecture~\ref{Conj:Class} is known to hold.

\subsection*{Acknowledgements}
We would like to thank Melody Chan, Hannah Larson, and Yoav Len, who each provided helpful insights during their visits to the University of Kentucky in 2019--2020.  We thank Yoav Len, Sam Payne, and Dhruv Ranganathan for helpful comments on an earlier draft of this paper, and Gavril Farkas for telling us about his paper \cite{FarkasRimanyi}, which was the inspiration for Example~\ref{Ex:Quadric}.  We thank Eric Larson, Hannah Larson, and Isabel Vogt for pointing us toward the existing literature on $k$-cores and the affine symmetric group.  This work was supported by NSF DMS-1601896.

\section{Preliminaries}

\subsection{Partitions and Tableaux}

Throughout, we use the convention that $\mathbb{N}$ denotes the positive integers.  By a slight abuse of terminology, we use the term \emph{partition} to refer to the Ferrers diagram of a partition.

\begin{definition}
A \emph{partition} is a finite subset $\lambda \subset \mathbb{N}^2$ with the property that, if $(x,y) \in \lambda$, then 
\begin{enumerate}
\item either $x=1$ or $(x-1,y) \in \lambda$, and 
\item  either $y=1$ or $(x,y-1) \in \lambda$.
\end{enumerate}
\end{definition}

It is standard to depict a partition as a set of boxes, with a box in position $(x,y)$ if $(x,y) \in \lambda$.  We follow the English convention, so that the box $(1,1)$ appears in the upper lefthand corner.  Given a partition $\lambda$, we define its \emph{transpose} to be
\[
\lambda^T := \{ (x,y) \in \mathbb{N}^2 \mbox{ } \vert \mbox{ } (y,x) \in \lambda \} .
\]
The corners of a partition will play an important role in our discussion.

\begin{definition}
Let $\lambda$ be a partition.  A box $(x,y) \in \lambda$ is called an \emph{inside corner} if $(x+1,y) \notin \lambda$ and $(x,y+1) \notin \lambda$.  A box $(x,y) \notin \lambda$ is called an \emph{outside corner} if
\begin{enumerate}
\item  either $x=1$ or $(x-1,y) \in \lambda$, and
\item  either $y=1$ or $(x,y-1) \in \lambda$.
\end{enumerate}
\end{definition}

In other words, a box $(x,y) \in \lambda$ is an inside corner if $\lambda \smallsetminus (x,y)$ is a partition, and a box $(x,y) \notin \lambda$ is an outside corner if $\lambda \cup (x,y)$ is a partition.  

Given a positive integer $g$, we write $[g]$ for the finite set $\{ 1, 2, \ldots , g \}$, and let ${{[g]}\choose{n}}$ denote the set of size-$n$ subsets of $[g]$. A \emph{tableau} on a partition $\lambda$ with alphabet $[g]$ is a function $t: \lambda \to [g]$ satisfying:
\begin{align*}
t(x,y) &> t(x,y-1) \mbox{ for all } (x,y) \in \lambda \mbox{ with } y > 1, \mbox{ and} \\
t(x,y) &> t(x-1,y) \mbox{ for all } (x,y) \in \lambda \mbox{ with } x > 1 .
\end{align*}
We depict a tableau by filling each box of $\lambda$ with an element of $[g]$.  The tableau condition is satisfied if the symbols in each row are increasing and the symbols in each column are increasing.  We write $YT(\lambda)$ for the set of tableaux on the partition $\lambda$.  Given a tableau $t$ on $\lambda$, we define its \emph{transpose} to be the tableau $t^T$ on $\lambda^T$ given by
\[
t^T (x,y) = t(y,x) \mbox{ for all } (x,y) \in \lambda^T .
\]

We will be primarily concerned with the combinatorics of certain special kinds of tableaux, called $k$-uniform displacement tableaux.

\begin{definition} \cite[Definition~2.5]{PfluegerkGonal}
\label{Def:Displacement}
A tableau $t$ on a partition $\lambda$ is called a \emph{$k$-uniform displacement tableau} if, whenever $t(x,y) = t(x',y')$, we have $y-x \equiv y'-x' \pmod{k}$.
\end{definition}

We write $YT_k (\lambda)$ for the set of $k$-uniform displacement tableaux on the partition $\lambda$.  The $k$-uniform displacement condition is satisfied if the lattice distance (or taxicab distance) between any two boxes containing the same symbol is a multiple of $k$.  For example, Figure~\ref{Fig:3Uniform} depicts a 3-uniform displacement tableau with alphabet $[5]$.  Note that the two boxes containing the symbol 3 have lattice distance 3, and any two of the three boxes containing the symbol 5 have lattice distance a multiple of 3.

\begin{figure}[H]
\begin{ytableau}
1 & 3 & 4 & 5\\
2 & 5\\
3\\
5\\
\end{ytableau}

\caption{A 3-uniform displacement tableau with alphabet $[5]$.}
\label{Fig:3Uniform}
\end{figure}

\subsection{Splitting Types}
\label{Sec:Split}

Let $\pi: C \to \PP^1$ be a branched cover of degree $k$, where the domain has genus $g$.  If $L$ is a line bundle on $C$, then its pushforward $\pi_* L$ is a vector bundle on $\PP^1$ of rank $k$.  Every such vector bundle splits as a direct sum of line bundles:
\[
\pi_* L = \cO (\mu_1 ) \oplus \cdots \oplus \cO (\mu_k) .
\]
Throughout, we order the integers $\mu_i$ from smallest to largest, i.e. $\mu_1 \leq \cdots \leq \mu_k$.  We refer to the vector $\bmu = (\mu_1 , \ldots , \mu_k ) \in \ZZ^k$ as the \emph{splitting type} of $L$, and write $\pi_* L \cong \cO (\bmu)$.  Many natural invariants of the line bundle $L$ are determined by its splitting type.
\begin{align}
\label{eq:twist}
h^0 (C, L \otimes \pi^* \cO (m)) &= x_m (\bmu) := \sum_{i=1}^k \max \{ 0, \mu_i + m + 1 \} \\
h^1 (C, L \otimes \pi^* \cO (m)) &= y_m (\bmu) := \sum_{i=1}^k \max \{ 0, -\mu_i - m - 1 \} \nonumber \\
\deg L = d (\bmu) &:= g-1 + \sum_{i=1}^k (\mu_i + 1) . \nonumber
\end{align}

We define the \emph{splitting type loci}
\begin{align*}
W^{\bmu} (C) &= \left\{ L \in \Pic (C) \mbox{ }  \vert \mbox{ } \pi_* L \cong \cO (\bmu) \right\} \\
\overline{W}^{\bmu} (C) &= \left\{ L \in \Pic^{d(\bmu)} (C) \mbox{ } \vert \mbox { } h^0 (C, L \otimes \pi^* \cO (m)) \geq x_m (\bmu) \mbox{ for all } m \right\} . 
\end{align*}

Equation (\ref{eq:twist}) above show that $W^{\bmu} (C)$ is contained in $\overline{W}^{\bmu} (C)$.  The strata $\overline{W}^{\bmu} (C)$ are closed, whereas the strata $W^{\bmu}(C)$ are locally closed.  It is not necessarily the case that $\overline{W}^{\bmu} (C)$ is the closure of $W^{\bmu} (C)$.  This is the case, however, when all splitting type loci have the expected dimension.  (See \cite[Lemma~2.1]{Larson}.)

The expected codimension of $\overline{W}^{\bmu} (C)$ in $\Pic^{d(\bmu)} (C)$ is given by the \emph{magnitude}
\[
\vert \bmu \vert := \sum_{i < j} \max \{ 0, \mu_j - \mu_i - 1 \} .
\]

A consequence of (\ref{eq:twist}) is that the sum of the $\ell$ largest entries of $\bmu$ is an upper semicontinuous invariant.  This defines a natural partial order on splitting types.  Specifically, given two splitting types $\bmu$ and $\blam$ such that $d(\bmu) = d(\blam)$, we say that $\bmu \leq \blam$ if
\[
\mu_1 + \cdots + \mu_{\ell} \leq \lambda_1 + \cdots + \lambda_{\ell} \mbox{ for all } \ell \leq k .
\]
If one considers a splitting type to be a partition of $d(\bmu)$ with possibly negative parts, then this partial order is the usual dominance order on partitions.  This partial order has the following interpretation.

\begin{lemma}
\label{lem:SplitTypeContainment}
If $\bmu \leq \blam$, then $x_m (\bmu) \geq x_m (\blam)$ for all $m$, hence $\overline{W}^{\bmu}(C) \subseteq \overline{W}^{\blam}(C)$.
\end{lemma}

\begin{proof}
Let $m$ be an integer and $J$ the minimal index such that $\lambda_J + m + 1 \geq 0$. 
Since $\bmu \leq \blam$, we have
\[
\mu_1 + \cdots + \mu_k= \lambda_1 + \cdots + \lambda_k
\] 
\[
\mu_1 + \cdots + \mu_{J-1} \leq \lambda_1 + \cdots + \lambda_{J-1},
\]
which together imply that
\[
\mu_J+ \cdots + \mu_k \geq \lambda_J + \cdots + \lambda_k.
\]
Hence
\begin{align*}
\sum_{i=1}^k \max \lbrace 0, \lambda_i + m + 1 \rbrace &= ( \lambda_J + m + 1 ) + \cdots + ( \lambda_k + m + 1 ) \\
&\leq ( \mu_J + m + 1 ) + \cdots + ( \mu_k + m + 1 ) \leq \sum_{i=1}^k \max \lbrace 0, \mu_i + m + 1 \rbrace .
\end{align*}
\end{proof}

\subsection{Chains of Loops}
\label{Sec:COL}

We briefly discuss the theory of special divisors on chains of loops from \cite{CDPR, PfluegerkGonal, PfluegerCycles, JensenRanganthan}.  For a broader review of divisors on tropical curves, we refer the reader to \cite{Baker, BakerJensen}.  For our purposes however, we will only require the material surveyed here.

Throughout, we let $\Gamma$ be a chain of $g$ loops with bridges, as pictured in Figure~\ref{Fig:ChainOfLoops}.  Let $m_j$ be the length of the bottom edge of the $j$th loop and $\ell_j$ the length of the top edge.

\begin{figure}[h!]
\begin{tikzpicture}

\draw (0,0) circle (1);
\draw (1,0)--(2,0);
\draw (3,0) circle (1);
\draw (4,0)--(5,0);
\draw (6,0) circle (1);
\draw (7,0)--(8,0);
\draw (9,0) circle (1);
\draw (10,0)--(11,0);
\draw (12,0) circle (1);

\draw [<->] (7.15,0.5) arc[radius = 1.15, start angle=10, end angle=170];
\draw [<->] (7.15,-0.5) arc[radius = 1.15, start angle=-9, end angle=-173];

\draw (6,1.75) node {\footnotesize$\ell_j$};
\draw (6,-1.75) node {\footnotesize$m_j$};

\draw [ball color=black] (-1,0) circle (0.5mm);
\draw [ball color=black] (1,0) circle (0.5mm);

\draw [ball color=black] (2,0) circle (0.5mm);
\draw [ball color=black] (4,0) circle (0.5mm);

\draw [ball color=black] (5,0) circle (0.5mm);
\draw [ball color=black] (7,0) circle (0.5mm);

\draw [ball color=black] (8,0) circle (0.5mm);
\draw [ball color=black] (10,0) circle (0.5mm);

\draw [ball color=black] (11,0) circle (0.5mm);
\draw [ball color=black] (13,0) circle (0.5mm);
\end{tikzpicture}
\caption{The chain of loops $\Gamma$.}
\label{Fig:ChainOfLoops}
\end{figure}
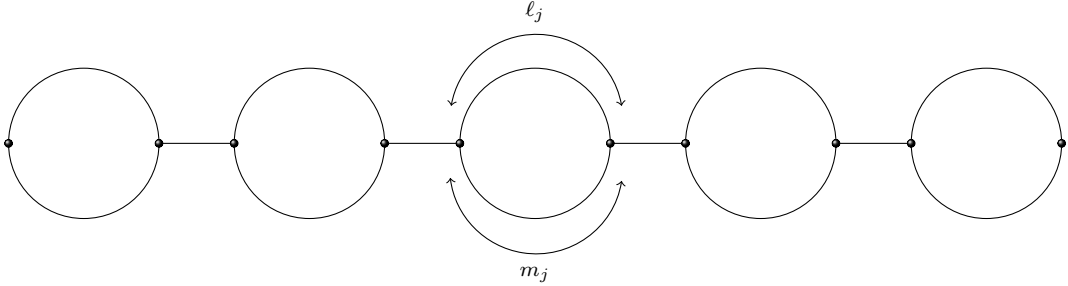

\begin{definition} \cite[Definition~1.9]{PfluegerCycles}
\label{Def:TorsionOrder}
If $\ell_j + m_j$ is an irrational multiple of $m_j$, then the $j$th \emph{torsion order} $\tau_j$ of $\Gamma$ is 0.  Otherwise, we define $\tau_j$ to be the minimum positive integer such that $\tau_j m_j$ is an integer multiple of $\ell_j + m_j$.  
\end{definition}

The \emph{$k$-gonal chain of loops of genus $g$} referred to in the introduction is the graph $\Gamma$ with the following torsion orders:

\begin{displaymath}
\tau_j := \left\{ \begin{array}{ll}
0 & \textrm{if $j<k$ or $j>g-k+1$} \\
k & \textrm{otherwise.}
\end{array}\right.
\end{displaymath}

\begin{remark}
Some of our arguments in Section~\ref{Sec:TSL} would be simplified if we assumed instead that $\tau_j = k$ for all $j$.  Although this does not affect the Brill-Noether theory of $\Gamma$, we prefer the torsion orders above because the family of chains of loops with these torsion orders has the same dimension as the Hurwitz space $\mathcal{H}_{g,k}$.
\end{remark}

The Jacobian of $\Gamma$ has two natural systems of coordinates.  The first uses the theory of break divisors from \cite{MikhalkinZharkov08,ABKS}.  On the $j$th loop, define $\langle \xi \rangle_j$ to be the point of distance $\xi m_j$ from the righthand vertex in the counterclockwise direction.  Every divisor class $D$ of degree $d$ has a unique \emph{break divisor} representative of the form
\[
(d-g)\langle 0 \rangle_g + \sum_{j=1}^g \langle \xi_j (D) \rangle_j . 
\]
Because this representative is unique, the functions $\xi_j \colon \Pic^d (\Gamma) \to \mathbb{R} / \left( \frac{m_j + \ell_j}{m_j} \right) \mathbb{Z}$ act as a system of coordinates on $\Pic^d (\Gamma)$.

Alternatively, define an orientation on $\Gamma$ by orienting each of the loops counterclockwise, and let $\omega_j$ be the harmonic 1-form supported on the $j$th loop with weight 1.  Given a divisor class $D$ on $\Gamma$, define
\[
\widetilde{\xi}_j (D) := \frac{1}{m_j} \int_{\langle 0 \rangle_g}^D \omega_j .
\]
By the tropical Abel-Jacobi theorem \cite{BakerNorine}, since the set of 1-forms $\omega_1 , \ldots , \omega_g$ is a basis for $\Omega (\Gamma)$, the functions $\widetilde{\xi}_j \in \Omega (\Gamma)^* / H_1 (\Gamma, \mathbb{Z})$ form a system of coordinates on $\Jac (\Gamma)$.  In our combinatorial arguments, we tend to use the functions $\xi_j$ more often that $\widetilde{\xi}_j$, but the latter are useful due to their linearity.  That is, $\widetilde{\xi}_j (D_1 + D_2) = \widetilde{\xi}_j (D_1) + \widetilde{\xi}_j (D_2)$.

It is straightforward to translate between the two systems of coordinates.  Specifically, we have $\widetilde{\xi}_j (D) = \xi_j (D) - (j-1)$.  Since $\widetilde{\xi}_j$ is linear, it follows that
\begin{align}
\label{eq:linear}
\xi_j (D_1 + D_2) = \xi_j (D_1) + \widetilde{\xi}_j (D_2) .
\end{align}

In \cite{PfluegerCycles}, Pflueger classifies the special divisor classes on $\Gamma$.  This classification specializes to the ``generic'' case where $k = \lfloor \frac{g+3}{2} \rfloor$, studied in \cite{CDPR}.

\begin{definition} \cite[Definition~3.5]{PfluegerCycles}
\label{Def:Torus}
Let $a$ and $b$ be positive integers and let $\lambda$ be the rectangular partition
\[
\lambda = \{ (x,y) \in \mathbb{N}^2 \mbox{ } \vert \mbox{ } x \leq a, \mbox{ } y \leq b \} .
\]
Given a $k$-uniform displacement tableau $t$ on $\lambda$ with alphabet $[g]$, we define $\TT(t)$ as follows.
\[
\TT (t) := \{ D \in \Pic^{g+a-b-1} (\Gamma) \mbox{ } \vert \mbox{ } \xi_{t(x,y)} (D) = y-x \} .
\]
\end{definition}

In the system of coordinates $\xi_j$, $\TT (t)$ is a coordinate subtorus, where the coordinate $\xi_j$ is fixed if and only if the symbol $j$ is in the image of $t$.  The codimension of $\TT(t)$ is therefore equal to the number of distinct symbols in $t$.  If the symbol $j$ appears in multiple boxes of the tableau $t$, then the $k$-uniform displacement condition guarantees that the two boxes impose the same condition on $\xi_j$.

\begin{theorem} \cite[Theorem~1.4]{PfluegerCycles}
\label{thm:Classification}
For any positive integers $r$ and $d$ satisfying $r > d-g$, we have
\[
W^r_d (\Gamma) = \bigcup \TT (t),
\]
where the union is over $k$-uniform displacement tableaux on $[r+1]\times[g-d+r]$ with alphabet $[g]$.
\end{theorem}

A consequence of Theorem~\ref{thm:Classification} is that $\Gamma$ has a unique divisor class of degree $k$ and rank 1, which we denote by $g^1_k$.  This justifies the terminology that $\Gamma$ is a $k$-gonal chain of loops.  Specifically, the unique $k$-uniform displacement tableau on $[2] \times [g-k+1]$ with alphabet $[g]$ contains the symbols $1, 2, \ldots , g-k+1$ in the first column and the symbols $k, k+1, \ldots , g$ in the second column.  In particular, we have
\begin{displaymath}
\widetilde{\xi}_j (g^1_k) = \left\{ \begin{array}{ll}
0 & \textrm{if $j \leq g-k+1$} \\
k & \textrm{if $j > g-k+1$.}
\end{array} \right.
\end{displaymath}

Given a splitting type $\bmu \in \mathbb{Z}^k$, we define the tropical splitting type locus
\[
\overline{W}^{\bmu} (\Gamma) = \left\{ D \in \Pic^{d(\bmu)} (\Gamma) \mbox{ } \vert \mbox { } \rk (D + mg^1_k) \geq x_m (\bmu) -1 \mbox{ for all } m \right\} .
\]
Note that the tropical splitting type locus can be defined in this way for \emph{any} tropical curve $\Gamma$ with a distinguished $g^1_k$.  By Lemma~\ref{lem:SplitTypeContainment}, if $\bmu \leq \blam$, then $\overline{W}^{\bmu} (\Gamma) \subseteq \overline{W}^{\blam} (\Gamma)$.  The following is a straightforward consequence of Baker's Specialization Lemma.

\begin{proposition}
\label{prop:specialization}
Let $C$ be a curve of genus $g$ and gonality $k$ over a nonarchimedean field $K$ with skeleton $\Gamma$.  Then
\[
\Trop \left( \overline{W}^{\bmu} (C) \right) \subseteq \overline{W}^{\bmu} (\Gamma) .
\]
\end{proposition}

\begin{proof}
Since the divisor of degree $k$ and rank 1 on $\Gamma$ is unique, it must be the tropicalization of the $g^1_k$ on $C$ by Baker's Specialization Lemma.  If $D \in \overline{W}^{\bmu}(C)$, then by definition we have
\[
h^0 (C,D + mg^1_k) \geq x_m (\bmu) \mbox{} \text{ for all } m .
\]
By Baker's Specialization Lemma, we have
\[
\rk (\Trop (D + m g^1_k)) \geq h^0 (C,D + mg^1_k) -1 \geq x_m (\bmu) -1 \mbox{} \text{ for all } m .
\]
Thus, $\Trop (D) \in \overline{W}^{\bmu} (\Gamma)$.
\end{proof}

\section{Tropical Splitting Type Loci}
\label{Sec:TSL}

In this section, we prove Theorem~\ref{thm:TropClassify}, which gives an explicit description of splitting type loci on a $k$-gonal chain of loops.  Before proving Theorem~\ref{thm:TropClassify}, we first define a partition $\lambda (\bmu)$ associated to each splitting type $\bmu$.

\subsection{Staircases}

\begin{definition}\label{def:SplitTypePartition}
Given a splitting type $\bmu \in \ZZ^k$ and an integer $m$, we define the rectangular partition
\[
\lambda_m (\bmu) := \left\{ (x,y) \in \mathbb{N}^2 \mbox{ } \vert \mbox{ } x \leq x_m (\bmu), \mbox{ } y \leq y_m (\bmu) \right\}.
\]
We further define
\begin{align*}
\lambda (\bmu) &= \bigcup_{m \in \mathbb{Z}} \lambda_m (\bmu) \\
&= \left\{ (x,y) \in \mathbb{N}^2 \mbox{ } \vert \mbox{ } \exists m \in \ZZ \mbox{ s.t. } x \leq x_m (\bmu), \mbox{ } y \leq y_m (\bmu) \right\}.
\end{align*}
We call a partition of the form $\lambda (\bmu)$ a \emph{$k$-staircase}.
\end{definition}

\begin{example}
Let $\bmu = (-3,-1,1)$.  Figure~\ref{Fig:Parts} depicts the rectangular partitions $\lambda_{-1} (\bmu)$, $\lambda_0 (\bmu)$, and $\lambda_1 (\bmu)$, together with $\lambda (\bmu)$.  Note that $\lambda_m (\bmu)$ is empty for all $m$ other than $-1$, 0, or 1.

\begin{figure}[H]
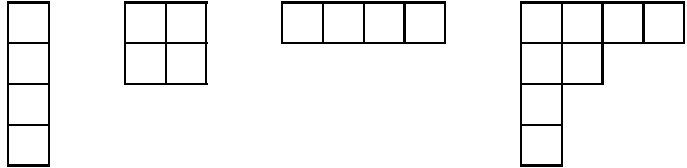


\begin{ytableau}
{} \\
{} \\
{} \\
{}
\end{ytableau}
\hspace{.3in}
\begin{ytableau}
{} & {} \\
{} & {}\\
\end{ytableau}
\hspace{.3in}
\begin{ytableau}
{} & {} & {} & {}\\
\end{ytableau}
\hspace{.3in}
\begin{ytableau}
{} & {} & {} & {}\\
{} & {}\\
{}\\
{}\\
\end{ytableau}

\caption{The partitions $\lambda_{-1} (\bmu)$, $\lambda_0 (\bmu)$, $\lambda_1 (\bmu)$, and $\lambda (\bmu)$, where $\bmu = (-3,-1,1)$.}
\label{Fig:Parts}
\end{figure}
\end{example}

\begin{remark}
If $\bmu = (\mu_1 , \ldots , \mu_k)$ and $\bmu' = (\mu_1 + m , \ldots , \mu_k + m )$ for some $m \in \ZZ$, then there is an isomorphism between $\overline{W}^{\bmu} (C)$ and $\overline{W}^{\bmu'}(C)$, given by twisting by $\pi^* \cO (m)$.  Correspondingly, we have $\lambda (\bmu) = \lambda (\bmu')$.   
\end{remark}

\begin{remark}
If $\bmu \leq \bmu'$, then by Lemma~\ref{lem:SplitTypeContainment}, we have $\lambda (\bmu') \subseteq \lambda (\bmu)$.
\end{remark}

If both $x_m (\bmu)$ and $y_m (\bmu)$ are positive, then the box $(x_m (\bmu),y_m (\bmu))$ is the unique inside corner of the rectangular partition $\lambda_m (\bmu)$, and one of the inside corners of $\lambda(\bmu)$.  We define
\[
\alpha_m (\bmu) = x_m (\bmu) - x_{m-1} (\bmu).
\]
Note that $\alpha_m (\bmu) \leq \alpha_{m+1} (\bmu)$ for all $m$, and $y_{m-1} (\bmu) - y_m (\bmu) = k - \alpha_m (\bmu)$.  We say that an integer $\alpha$ is a \emph{rank jump} in $\lambda(\bmu)$ if $\alpha = \alpha_m (\bmu)$ for some integer $m$.  We say that $\alpha$ is a \emph{strict rank jump} in $\lambda (\bmu)$ if $\alpha = \alpha_m (\bmu)$ for some integer $m$ such that both $x_{m-1} (\bmu)$ and $y_m (\bmu)$ are positive.  In other words, the strict rank jumps are $\alpha_{1-\mu_k} (\bmu), \alpha_{2-\mu_k} (\bmu), \ldots , \alpha_{-2-\mu_1} (\bmu)$.

\subsection{Tropical Splitting Type Loci}

We now define the analogue of the coordinate tori from \cite{PfluegerCycles}.

\begin{definition}
\label{Def:SplittingTorus}
Let $\bmu \in \ZZ^k$ be a splitting type.  Given an integer $m$ and a $k$-uniform displacement tableau $t$ on $\lambda (\bmu)$ with alphabet $[g]$, let $t_m$ denote the restriction of $t$ to the rectangular subpartition $\lambda_m (\bmu)$.  We define the coordinate subtorus $\TT(t)$ as follows.
\[
\TT (t) = \left\{ D \in \Pic^{d(\bmu)} (\Gamma) \mbox{ } \vert \mbox{ } D + mg^1_k \in \TT (t_m) \mbox{ for all } m \right\} .
\]
\end{definition}

From the definition it appears that, if one wants to determine whether a divisor class $D$ is contained in $\TT(t)$, one has to compute $\xi_j (D + mg^1_k)$ for all integers $m$.  Using (\ref{eq:linear}), however, we can simplify Definition~\ref{Def:SplittingTorus} as follows. 

\begin{lemma}
\label{Lem:AltChar}
Let $\bmu \in \ZZ^k$ be a splitting type, and let $t$ be a $k$-uniform displacement tableau on $\lambda (\bmu)$ with alphabet $[g]$.  Define the function
\begin{displaymath}
Z(x,y) = \left\{ \begin{array}{ll}
y-x & \textrm{if $t(x,y) \leq g-k+1$} \\
y-x+mk & \textrm{if $t(x,y) > g-k+1$ and $x_{m-1} (\bmu) < x \leq x_m (\bmu)$.}
\end{array} \right.
\end{displaymath}
Then
\[
\TT (t) := \{ D \in \Pic^{d(\bmu)} (\Gamma) \mbox{ } \vert \mbox{ } \xi_{t(x,y)} (D) = Z(x,y) \} .
\]
\end{lemma}

\begin{proof}
Let $m$ be an integer and let $t_m (x,y) = j$.  If $j \leq g-k+1$, then $\widetilde{\xi_j} (g^1_k) = 0$, and by (\ref{eq:linear}) we see that for any divisor class $D$ we have
\[
\xi_j (D) = \xi_j (D + mg^1_k) .
\]
It follows that $\xi_j (D) = y-x$ if and only if $\xi_j (D + mg^1_k ) = y-x$.

On the other hand, if $j > g-k+1$, then we must first show that $x_{m-1} (\bmu) < x \leq x_m (\bmu)$.  The second inequality follows from the fact that $(x,y) \in \lambda_m (\bmu)$.  If $x \leq x_{m-1} (\bmu)$, then the $k+1$ boxes in the hook
\[
H_m = \{ (x,y) \in \lambda (\bmu) \mbox{ } \vert \mbox{ } x \geq x_{m-1} (\bmu), y \geq y_m (\bmu) \}
\]
are all below and to the right of $(x,y)$.  The two inside corners $(x_{m-1} (\bmu), y_{m-1} (\bmu))$ and $(x_m (\bmu), y_m (\bmu))$ have lattice distance $k$, so they are the only two boxes of $H_m$ that can contain the same symbol.  It follows that $H_m$ contains at least $k$ distinct symbols greater than or equal to $j$.  Since $j > g-k+1$, this is impossible, hence $x > x_{m-1} (\bmu)$.  Now, since $\widetilde{\xi_j} (g^1_k) = k$, by (\ref{eq:linear}) we see that for any divisor class $D$ we have
\[
\xi_j (D) = \xi_j (D + mg^1_k) - mk.
\]
It follows that $\xi_j (D) = y-x+mk$ if and only if $\xi_j (D + mg^1_k ) = y-x$.
\end{proof}

As in Definition~\ref{Def:Torus}, the $k$-uniform displacement condition guarantees that, if the symbol $j$ appears in more than one box, then the boxes impose the same condition on $\xi_j$.  In particular, if $j > g-k+1$ and $t_m (x,y) = t_{m'} (x',y') = j$, then the $k$-uniform displacement condition guarantees that
\[
(y'-x') - (y-x) = (m-m')k ,
\]
so $Z(x,y) = Z(x',y')$.  As a consequence, we see that the codimension of $\TT(t)$ is equal to the number of distinct symbols in $t$.

\begin{example}
Figure~\ref{Fig:Divisor} depicts a 3-uniform displacement tableau $t$ on $\lambda (\bmu)$, where $\bmu = (-3,-1,1)$.  Since the tableau contains $g=5$ distinct symbols, $\TT (t)$ is a zero-dimensional torus.  In other words, it consists of a single divisor class $D$, also depicted in Figure~\ref{Fig:Divisor}.  In this picture, the chips on loops 2 and 4 are located at the points $\langle 1 \rangle_2$ and $\langle 1 \rangle_4$.  By Theorem~\ref{thm:TropClassify}, the divisor class $D$ is in $\overline{W}^{\bmu} (\Gamma)$.  That is, $D - g^1_3$ has rank 0, $D$ has rank 1, and $D + g^1_3$ has rank 3.

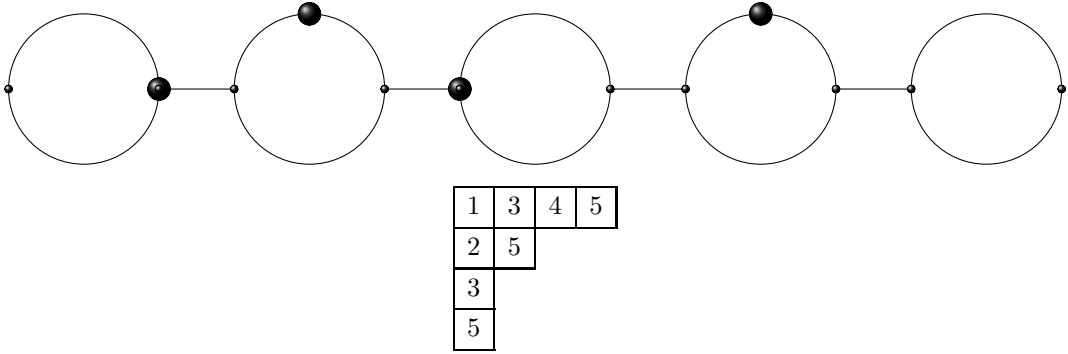
\begin{figure}[h!]
\begin{tikzpicture}

\draw (0,0) circle (1);
\draw (1,0)--(2,0);
\draw (3,0) circle (1);
\draw (4,0)--(5,0);
\draw (6,0) circle (1);
\draw (7,0)--(8,0);
\draw (9,0) circle (1);
\draw (10,0)--(11,0);
\draw (12,0) circle (1);

\draw [ball color=black] (1,0) circle (1.5mm);
\draw [ball color=black] (3,1) circle (1.5mm);
\draw [ball color=black] (5,0) circle (1.5mm);
\draw [ball color=black] (9,1) circle (1.5mm);

\draw [ball color=black] (-1,0) circle (0.5mm);
\draw [ball color=black] (1,0) circle (0.5mm);

\draw [ball color=black] (2,0) circle (0.5mm);
\draw [ball color=black] (4,0) circle (0.5mm);

\draw [ball color=black] (5,0) circle (0.5mm);
\draw [ball color=black] (7,0) circle (0.5mm);

\draw [ball color=black] (8,0) circle (0.5mm);
\draw [ball color=black] (10,0) circle (0.5mm);

\draw [ball color=black] (11,0) circle (0.5mm);
\draw [ball color=black] (13,0) circle (0.5mm);
\end{tikzpicture}

\vspace{.1in}
\begin{ytableau}
1 & 3 & 4 & 5\\
2 & 5\\
3\\
5\\
\end{ytableau}
\caption{A 3-uniform displacement tableau on $\lambda (-3,-1,1)$ and the corresponding divisor class.}
\label{Fig:Divisor}
\end{figure}
\end{example}

Lemma~\ref{Lem:AltChar} allows us to formulate the following analogue of \cite[Lemma~3.6]{CPJ}.
\newline
\begin{lemma}
\label{lem:contain}
Let $\bmu \in \ZZ^k$ be a splitting type, and let $t$, $t'$ be $k$-uniform displacement tableaux on $\lambda (\bmu)$.  Then $\TT (t) \subseteq \TT (t')$ if and only if
\begin{enumerate}
\item  every symbol in $t'$ is a symbol in $t$, and
\item  if $t(x,y) = t'(x',y')$, then $y-x \equiv y'-x' \pmod{k}$.
\end{enumerate}
\end{lemma}

We now prove Theorem~\ref{thm:TropClassify}.

\begin{proof}[Proof of Theorem~\ref{thm:TropClassify}]
We first show that 
\[
\overline{W}^{\bmu}(\Gamma) \supseteq \bigcup \TT(t) .
\]
Let $t$ be a $k$-uniform displacement tableau on $\lambda (\bmu)$, and let $D \in \TT(t)$.  By definition, $D + mg^1_k \in \TT(t_m)$ for all $m$.  It follows from Theorem~\ref{thm:Classification} that $D + mg^1_k$ has degree $d(\bmu) + mk$ and rank at least $x_m (\bmu) -1$ for all $m$.  By definition, we see that $D \in \overline{W}^{\bmu}(\Gamma)$.

We now show that
\[
\overline{W}^{\bmu}(\Gamma) \subseteq \bigcup \TT(t) .
\]
Let $D \in \overline{W}^{\bmu}(\Gamma)$.  By definition, $D + mg^1_k$ has degree $d(\bmu) + mk$ and rank at least $x_m (\bmu) - 1$ for all $m$.  By Theorem~\ref{thm:Classification}, there exists a $k$-uniform displacement tableau $t_m$ on the rectangular partition $\lambda_m (\bmu)$ such that $D + mg^1_k \in \TT(t_m)$. We construct a tableau $t$ on $\lambda(\bmu)$ as follows.  For each box $(x,y)$ in $\lambda(\bmu)$, define 
\[
t(x,y) = \min_{\substack{m \in \ZZ \text{ s.t.} \\ (x,y) \in \lambda_m (\bmu)}} t_m (x,y) .
\]
We first show that $t$ is a tableau on $\lambda (\bmu)$.  To see that $t$ is strictly increasing across rows, suppose that $x > 1$ and $t(x,y) = t_m (x,y)$.  Since $(x,y) \in \lambda_m (\bmu)$, we see that $(x-1,y) \in \lambda_m (\bmu)$ as well.  It follows that
\[
t(x-1,y) \leq t_m (x-1,y) < t_m (x,y) = t(x,y) .
\]
The same argument shows that $t$ is strictly increasing down the columns.

We now show that the tableau $t$ satisfies the $k$-uniform displacement condition.  Suppose that $t(x,y) = t(x',y')$.  By construction, there exist integers $m$ and $m'$ such that $t(x,y) = t_m (x,y)$ and $t(x',y') = t_{m'} (x',y')$.  Since $D + mg^1_k \in \TT(t_m)$ and $D + m'g^1_k \in \TT(t_{m'})$, we see that 
\begin{align*}
\xi_{t(x,y)} ( D + mg^1_k ) &\equiv y-x \pmod{k} \\
\xi_{t(x,y)} ( D + m'g^1_k ) &\equiv y'-x' \pmod{k}.
\end{align*}
It therefore suffices to show that 
\[
\xi_j (D + mg^1_k) \equiv \xi_j (D + m'g^1_k) \mbox{ for all } j.
\]
This follows from (\ref{eq:linear}) and the fact that $\widetilde{\xi}_j (g^1_k) \equiv 0 \pmod{k}$ for all $j$.

Finally, we show that $D \in \TT (t)$.  For every box $(x,y) \in \lambda (\bmu)$, there is an integer $m$ such that $\xi_{t(x,y)} (D + mg^1_k ) = y-x$.  By Lemma~\ref{Lem:AltChar}, we have $\xi_{t(x,y)} (D) = Z(x,y)$.  Since this holds for all $(x,y) \in \lambda (\bmu)$, we see that $D \in \TT (t)$ by Lemma~\ref{Lem:AltChar}.
\end{proof}

\subsection{Operations on Splitting Types}

Several operations on splitting types have simple interpretations in terms of the corresponding partitions.  The first of these corresponds to Serre duality.

\begin{lemma}
\label{lem:SerreDualPartition}
Let $\bmu = (\mu_1, \ldots , \mu_k )$ be a splitting type, and let $\bmu^T = (-\mu_k, \ldots , -\mu_1)$.  Then $\lambda(\bmu^T) = \lambda(\bmu)^T$.
\end{lemma}

\begin{proof}
Since both operations are involutions, it suffices to show that $\lambda(\bmu)^T \subseteq \lambda(\bmu^T)$.  Let $(x,y) \in \lambda(\bmu)$.  Then there exists an integer $m \in \ZZ$ such that 
\[
x \leq \sum_{i=1}^k \max \{ 0, \mu_i + m + 1 \},
\]
\[
y \leq \sum_{i=1}^k \max \{ 0, -\mu_i - m - 1 \}.
\]
Setting $m' = -2-m$, we see that
\begin{align*}
y \leq \sum_{i=1}^k \max \{ 0, -\mu_i - m - 1 \} &= \sum_{i=1}^k \max \{ 0, - \mu_i + m' + 1 \} \\
x \leq \sum_{i=1}^k \max \{ 0, \mu_i + m + 1 \} &= \sum_{i=1}^k \max \{ 0, \mu_i - m' - 1 \} .
\end{align*}
Thus, $(y,x) \in \lambda(\bmu^T)$. 
\end{proof}

As a consequence, we see that the set of partitions of the form $\lambda (\bmu)$ is closed under transpose.  We now show that it is also closed under the operations of deleting the top row or the leftmost column.

\begin{lemma}
\label{Lem:DeleteRow}
Let $\bmu \in \ZZ^k$ be a splitting type, let $s$ be the minimal index such that $\mu_s < \mu_{s+1}$, and let
\[
\bmu^{+} = (\mu_1 , \ldots , \mu_{s-1} , \mu_s + 1, \mu_{s+1}, \ldots , \mu_k ).
\]
Then $\lambda(\bmu^{+})$ is the partition obtained from $\lambda(\bmu)$ by deleting the first row.  Moreover, $\vert \bmu \vert - \vert \bmu^{+} \vert$ is equal to the largest strict rank jump in $\lambda(\bmu)$.

Similarly, let $s'$ be the maximal index such that $\mu_{s'} > \mu_{s'-1}$, and let
\[
\bmu^{-} = (\mu_1 , \ldots , \mu_{s'-1} , \mu_{s'} - 1, \mu_{s'+1}, \ldots , \mu_k ).
\]
Then $\lambda(\bmu^{-})$ is the partition obtained from $\lambda(\bmu)$ by deleting the leftmost column.  Moreover, $\vert \bmu \vert - \vert \bmu^{+} \vert$ is equal to $k-\alpha$, where $\alpha$ is the smallest strict rank jump in $\lambda(\bmu)$.
\end{lemma}

\begin{proof}
We prove the statements about $\bmu^+$.  The statements about $\bmu^-$ follow from Lemma~\ref{lem:SerreDualPartition}, together with the observation that $\bmu^- = (\bmu^{T^+)^T}$.  Let $(x,y) \in \lambda (\bmu^+)$.  Then there exists an integer $m$ such that
\begin{align*}
x & \leq \sum_{i=1}^k \max \{ 0, \mu^+_i + m + 1 \} \mbox{ and } \\
y & \leq \sum_{i=1}^k \max \{ 0, -\mu^+_i - m - 1 \} .
\end{align*}
Since $y$ is positive and $\mu_s$ is minimal, we see that $m \leq -2-\mu_s$.  It follows that $\mu^+_s + m + 1 \leq 0$, so
\begin{align*}
x \leq \sum_{i=1}^k \max \{ 0, \mu^+_i + m + 1 \} &= \sum_{i=1}^k \max \{ 0, \mu_i + m + 1 \} \\
y + 1 \leq 1 + \sum_{i=1}^k \max \{ 0, -\mu^+_i - m - 1 \} &= \sum_{i=1}^k \max \{ 0, -\mu_i - m - 1 \} .
\end{align*}
So $(x,y+1) \in \lambda (\bmu)$.  An analogous argument shows that, if $(x,y) \in \lambda (\bmu)$, then either $y=1$ or $(x,y-1) \in \lambda (\bmu^+)$.

We now compute $\vert \bmu \vert - \vert \bmu^+ \vert$.  If $i,j \neq s$, then $\mu^+_j - \mu^+_i = \mu_j - \mu_i$.  If $i < s$, then $\mu^+_s - \mu^+_i - 1 = 0$.  Finally, if $j > s$, then $\mu^+_j - \mu^+_s = \mu_j - \mu_s - 1$.  Thus,
\begin{align*}
\vert \bmu \vert - \vert \bmu^+ \vert &= \sum_{i<j} \Big( \max \{ 0, \mu_j - \mu_i \} - \max \{ 0, \mu^+_j - \mu^+_i - 1 \} \Big) \\
&= \sum_{j=1}^k \Big( \max \{ 0, \mu_j - \mu_s - 1 \} - \max \{ 0, \mu_j - \mu_s - 2 \} \Big) .
\end{align*}
On the other hand, the largest rank jump in $\lambda (\bmu)$ is
\[
\alpha_{-\mu_s - 2} (\bmu) = \sum_{j=1}^k \Big( \max \{ 0, \mu_j - \mu_s - 1 \} - \max \{ 0, \mu_j - \mu_s - 2 \} \Big) .
\]
\end{proof}

\section{Cores and Displacement}
\label{Sec:Partitions}

This section contains the main combinatorial arguments that will be used in our examination of tropical splitting type loci.  We study an operation on partitions known as \emph{displacement}, and a certain class of partitions known in the combinatorics literature as \emph{$k$-cores}, which includes the $k$-staircases.  Because of Theorem~\ref{thm:TropClassify}, we are interested in $k$-uniform displacement tableaux on partitions of this type.  A tableau $t$ on a partition $\lambda$ can be thought of as a chain of partitions
\[
\emptyset = \lambda_0 \subseteq \lambda_1 \subseteq \cdots \subseteq \lambda_n = \lambda,
\]
where $\lambda_j = \{ (x,y) \in \lambda \vert t(x,y) \leq j \}$.  This observation naturally leads us to study posets of partitions, where the cover relations guarantee that the resulting tableaux satisfy $k$-uniform displacement.

\subsection{Diagonals and Displacement}

Following \cite{CLRW}, given $a \in \mathbb{Z}/k\mathbb{Z}$, we define the corresponding \emph{diagonal (mod $k$)} to be
\[
D_a := \{ (x,y) \in \mathbb{N}^2 \mbox{ } \vert \mbox{ } y-x \equiv a \pmod{k} \} .
\]

\begin{definition} \cite[Definition~6.1]{PfluegerNegativeBN}
Let $\lambda$ be a partition.  The \emph{upward displacement}\footnote{This terminology is consistent with \cite{PfluegerNegativeBN}.  In that paper, partitions are depicted according to the French convention, whereas ours are in the English style.  Because of this, the \emph{upward} displacement adds boxes \emph{below} the partition.} of $\lambda$ with respect to $a \in \mathbb{Z}/k\mathbb{Z}$ is the partition $\lambda_a^+$ obtained from $\lambda$ by adding all outside corners in $D_a$.

Similarly, the \emph{downward displacement} of $\lambda$ with respect to $a \in \mathbb{Z}/k\mathbb{Z}$ is the partition $\lambda_a^-$ obtained from $\lambda$ by deleting all inside corners in $D_a$.
\end{definition}

\begin{example}
The operations of upward displacement and downward displacement are not inverses.  For example, consider the partition $\lambda$ on the left in Figure~\ref{Fig:UpDown}, where each box has been decorated with its diagonal (mod 4).  The second partition in the figure is $\lambda^+_2$, the upward displacement with respect to 2 (mod 4), and the third partition is $(\lambda^+_2)^-_2$, the downward displacement of the second partition, again with respect to 2 (mod 4).  Note that the first partition and the third partition do not agree.

\begin{figure}[H]
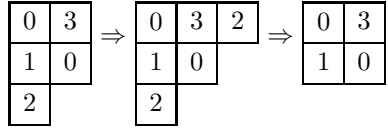

\begin{ytableau}
0 & 3\\
1 & 0\\
2\\
\end{ytableau}
$\Rightarrow$
\begin{ytableau}
0 & 3 & 2\\
1 & 0\\
2\\
\end{ytableau}
$\Rightarrow$
\begin{ytableau}
0 & 3\\
1 & 0\\
\end{ytableau}

\caption{Upward displacement followed by downward displacement does not necessarily yield the original partition.}
\label{Fig:UpDown}
\end{figure}

\end{example}

There are important examples of partitions for which the concatenation of an upward and a downward displacement \emph{is} the identity.

\begin{definition}
A partition $\lambda$ is called a $k$-\emph{core} if it can be obtained from the empty partition by a sequence of upward displacements with respect to congruence classes in $\mathbb{Z}/k\mathbb{Z}$.
\end{definition}

We write $\mathcal{P}_k$ for the poset of $k$-cores, where $\lambda' \leq \lambda$ if $\lambda$ can be obtained from $\lambda'$ by a sequence of upward displacements with respect to congruence classes in $\mathbb{Z}/k\mathbb{Z}$.  If $\lambda \in \mathcal{P}_k$, we write $\mathcal{P}_k (\lambda)$ for the interval (or principal order ideal) below $\lambda$ in $\mathcal{P}_k$.

\begin{example}

Figure~\ref{Fig:Interval} depicts a Hasse diagram for $\mathcal{P}_3 (\lambda (\bmu))$, where $\bmu=(-3,-1,1)$.  The diagram is drawn from left to right, rather than bottom to top, to preserve space on the page.  Note that $\lambda (\bmu)$ is a 3-core, and that every maximal chain in the interval below $\lambda (\bmu)$ has the same length.  As we shall see, the fact that the length of a maximal chain is 5 corresponds to the fact that any 3-uniform displacement tableau on $\lambda (\bmu)$ has at least 5 symbols.  The fact that there are 2 maximal chains corresponds to the fact that there are 2 such tableaux with alphabet $[5]$.

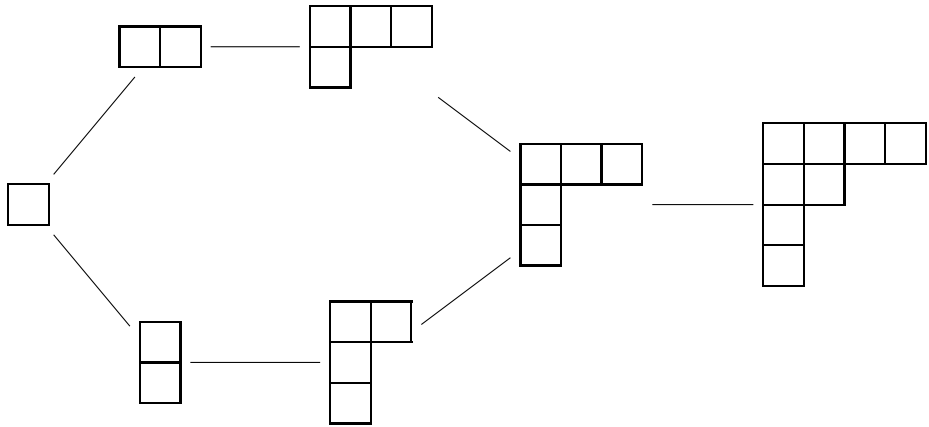
\begin{figure}[h]
\begin{tikzpicture}[scale=.7]
  \node (5) at (11,0) {\begin{ytableau}
{} & {} & {} & {}\\
{} & {}\\
{}\\
{}\\
\end{ytableau}};

  \node (4) at (6,0) {\begin{ytableau}
{} & {} & {}\\
{}\\
{}\\
\end{ytableau}};

  \node (3l) at (2,-3) {\begin{ytableau}
{} & {}\\
{}\\
{}\\
\end{ytableau}};

\node (3r) at (2,3) {\begin{ytableau}
{} & {} & {}\\
{}\\
\end{ytableau}};

 \node (2l) at (-2,-3) {\begin{ytableau}
{}\\
{}\\
\end{ytableau}};

  \node (2r) at (-2,3) {\begin{ytableau}
{} & {}\\
\end{ytableau}};

  \node (1) at (-4.5,0) {\begin{ytableau}
{}\\
\end{ytableau}};

  \draw (1) -- (2l) -- (3l) -- (4);
  
  \draw (5) -- (4) -- (3r) -- (2r) -- (1);
\end{tikzpicture}
\caption{A principal order ideal in $\mathcal{P}_3$.}
\label{Fig:Interval}

\end{figure}

\end{example}

\begin{remark}
Recall that, if $\bmu \leq \bmu'$, then $\lambda (\bmu') \subseteq \lambda (\bmu)$.  It is not necessarily true, however, that $\lambda (\bmu') \leq \lambda (\bmu)$ in the poset $\mathcal{P}_k$.  For example, if $\bmu = (-3,-1,1)$ and $\bmu' = (-3,0,0)$, then $\bmu \leq \bmu'$ but the partition $\lambda (\bmu')$, pictured in Figure~\ref{Fig:Boring}, is not contained in $\mathcal{P}_3 (\lambda (\bmu))$, pictured in Figure~\ref{Fig:Interval}.

\begin{figure}[H]
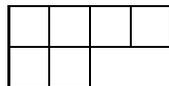

\begin{ytableau}
{} & {} & {} & {}\\
{} & {}\\
\end{ytableau}

\caption{The partition $\lambda (\bmu')$ is not in the principal order ideal of Figure~\ref{Fig:Interval}.}
\label{Fig:Boring}
\end{figure}
\end{remark}

We note the following simple observation.

\begin{lemma}
\label{Lem:Transpose}
The transpose of a $k$-core is a $k$-core.
\end{lemma}

\begin{proof}
This follows directly from the fact that $(\lambda^+_a)^T = (\lambda^T)^+_{-a}$.
\end{proof}

We now define some invariants of partitions.  Let $\lambda$ be a partition and $a \in \mathbb{Z}/k\mathbb{Z}$ a congruence class. We define
\[
C_a (\lambda) := \max \left\{ y \text{ } \vert \text{ } \exists (x,y) \in \lambda \cap D_a \text{ with } (x,y+1) \notin \lambda \right\} .
\]
In other words, $C_a (\lambda)$ is the height of the tallest column whose last box is in $D_a$.  If no such column exists, we define $C_a (\lambda)$ to be zero.  We write
\[
\bC (\lambda) = (C_0 (\lambda) , C_1 (\lambda) , \ldots , C_{k-1} (\lambda)),
\]
and further define
\[
\rho_k (\lambda) := \sum_{a \in \mathbb{Z}/k\mathbb{Z}} C_a (\lambda) .
\]

\begin{example}
Figure~\ref{Fig:RhoK} again depicts the partition $\lambda (\bmu)$, where $\bmu = (-3,-1,1)$.  Each column is labeled by the diagonal (mod 3) containing its last box.  The tallest column whose last box is in $D_0$ has height 4, the tallest column whose last box is in $D_1$ has height 1, and there is no column whose last box is in $D_2$.  Therefore, $\bC (\lambda (\bmu)) = (4,1,0)$, and
\[
\rho_3 (\lambda (\bmu)) = 4 + 1 + 0 = 5 = \vert \bmu \vert .  
\]

\begin{figure}[H]
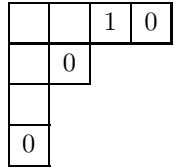

\begin{ytableau}
{} & {} & 1 & 0\\
{} & 0\\
{}\\
0\\
\end{ytableau}

\caption{The partition $\lambda (-3,-1,1)$, with each column labeled by the diagonal (mod 3) containing its last box.}
\label{Fig:RhoK}
\end{figure}
\end{example}

\subsection{Descent}

We now provide an alternate characterization of $k$-cores.  Most of the material in this and the next subsection has appeared previously in the literature on $k$-cores.  (See, for example, \cite{LapointeMorse, LLMSSZ}.)  We nevertheless include these arguments here, as they are fairly short and we wish to advertise these ideas.

\begin{definition}
We say that a partition $\lambda$ satisfies $k$-\emph{descent} if the following condition holds for every congruence class $a \in \mathbb{Z}/k\mathbb{Z}$.  Whenever $(x,y) \in \lambda \cap D_a$ and $(x+1,y) \notin \lambda$, then $C_{a-1} (\lambda) < y$.
\end{definition}

\begin{example}
The partition $\lambda$ pictured on the left in Figure~\ref{Fig:UpDown} does not satisfy 4-descent, because the last box in the first row is in $D_3$, and there exists a column whose last box is in $D_2$.  In other words, $(2,1) \in \lambda \cap D_3$ and $(3,1) \notin \lambda$, but $C_2 (\lambda) = 3 \geq 1$.

On the other hand, the partition $\lambda (\bmu)$ pictured in Figure~\ref{Fig:RhoK} does satisfy 3-descent.  There is no row whose last box is in $D_1$.  The last box in the first row is in $D_0$, and there is no column whose last box is in $D_2$.  The last box in the third row is in $D_2$, and $C_1 (\lambda (\bmu)) = 1 < 3$.
\end{example}

\begin{remark}
If $\lambda$ satisfies $k$-descent, then there is a congruence class $a \in \ZZ/k\ZZ$ such that $C_a (\lambda) = 0$.  Specifically, if $(x,1)$ is the last box in the first row, then by definition $C_{-x} (\lambda) = 0$.
\end{remark}

Our goal for this subsection is to prove the following.

\begin{proposition}
\label{Prop:Char}
A partition $\lambda$ is a $k$-core if and only if both $\lambda$ and $\lambda^T$ satisfy $k$-descent.
\end{proposition}

To prove Proposition~\ref{Prop:Char}, we will need a few preliminary results.  First, we examine the behavior of inside corners in partitions that satisfy $k$-descent.

\begin{lemma}
\label{lem:LeftmostCornerA}
Let $\lambda$ be a partition that satisfies $k$-descent, and let $a \in \mathbb{Z}/k\mathbb{Z}$ be a congruence class.  If $\lambda$ has an inside corner in $D_a$, then the tallest column whose last box is in $D_a$ contains an inside corner.
\end{lemma}

\begin{proof}
Let $(x,y) \in \lambda \cap D_a$ be an inside corner, and consider the tallest column whose last box is in $D_a$.  If it doesn't contain an inside corner, then the column immediately to the right has the same height, and its last box is in $D_{a-1}$.  But the height of this column is greater than $y$, contradicting the definition of $k$-descent.
\end{proof}

\begin{lemma}
\label{lem:CongAFurtherLeft}
Let $\lambda$ be a partition that satisfies $k$-descent.  Then $\lambda$ has an inside corner in $D_a$ if and only if $C_{a-1}(\lambda) < C_a (\lambda)$.
\end{lemma}

\begin{proof}
First, suppose that $\lambda$ has an inside corner in $D_a$.  By Lemma~\ref{lem:LeftmostCornerA}, the tallest column of $\lambda$ whose last box is in $D_a$ ends in an inside corner.  In other words, there is an $x$ such that $(x,C_a (\lambda)) \in \lambda \cap D_a$ and $(x+1,C_a(\lambda)) \notin \lambda$.  Thus, by the definition of $k$-descent, we see that $C_{a-1} (\lambda) < C_a (\lambda)$.

Conversely, suppose that $C_{a-1}(\lambda) < C_a (\lambda)$, and consider the tallest column of $\lambda$ whose last box is in $D_a$.  Let $(x,C_a (\lambda))$ be the last box in this column.  By definition, $(x,C_a (\lambda)+1) \notin \lambda$.  Since $C_{a-1}(\lambda) < C_a (\lambda)$ and $(x+1,C_a (\lambda)) \in D_{a-1}$, we see that $(x+1,C_a (\lambda)) \notin \lambda$.  Thus, $(x,C_a (\lambda)) \in D_a$ is an inside corner.
\end{proof}

\begin{lemma}
\label{lem:FindColLeft}
Let $\lambda$ be a partition, and suppose that both $\lambda$ and $\lambda^T$ satisfy $k$-descent.  For any congruence class $a \in \mathbb{Z}/k\mathbb{Z}$, $\lambda$ cannot have both an inside corner and an outside corner in $D_a$.
\end{lemma}

\begin{proof}
Suppose that $(x,y) \in D_a$ is an inside corner and $(x',y') \in D_a$ is an outside corner.  By definition, either $y'=1$ or $(x',y'-1) \in \lambda \cap D_{a-1}$, hence $C_{a-1}(\lambda) \geq y'-1$.  Since $\lambda$ satisfies $k$-descent, we see that $y'-1 < y$.  Similarly, since $\lambda^T$ satisfies $k$-descent, we see that $x'-1 < x$.  Together, these inequalities imply that $(x',y') \in \lambda$, contradicting our assumption that $(x',y')$ is an outside corner.
\end{proof}

Lemma~\ref{lem:FindColLeft} implies that, when restricted to partitions satisfying $k$-descent, the operations of upward and downward displacement are inverses.

\begin{lemma}
\label{lem:DisplacementsInvertible}
Let $\lambda$ be a partition, and suppose that both $\lambda$ and $\lambda^T$ satisfy $k$-descent.  If $\lambda$ has an inside corner in $D_a$, then $\lambda = (\lambda^-_a)^+_a$.  Similarly, if $\lambda$ has an outside corner in $D_a$, then $\lambda = (\lambda^+_a)^-_a$.
\end{lemma}

\begin{proof}
We show the first equality above.  The second equality follows from an analogous argument.  Note that $\lambda \subseteq (\lambda^-_a)^+_a$.  To see the reverse containment, let $(x,y) \in (\lambda^-_a)^+_a$.  If $(x,y) \notin D_a$ or $(x,y)$ is not an inside corner of $(\lambda^-_a)^+_a$, then $(x,y) \in \lambda^-_a \subset \lambda$.  On the other hand, if $(x,y) \in D_a$ is an inside corner of $(\lambda^-_a)^+_a$, then neither $(x-1,y)$ nor $(x,y-1)$ are in $D_a$, so either $x=1$ or $(x-1,y) \in \lambda$, and either $y=1$ or $(x,y-1) \in \lambda$.  It follows that either $(x,y) \in \lambda$ or $(x,y)$ is an outside corner of $\lambda$.  By Lemma~\ref{lem:FindColLeft}, however, $\lambda$ cannot have an outside corner in $D_a$.  Thus, $(x,y) \in \lambda$, and $(\lambda^-_a)^+_a \subseteq \lambda$.
\end{proof}

Crucially, the $k$-descent property is preserved by upward and downward displacements.

\begin{lemma}
\label{Lem:RowColCondition}
Let $\lambda$ be a partition that satisfies $k$-descent.  Then, for any $a \in \mathbb{Z}/k\mathbb{Z}$, $\lambda_a^+$ and $\lambda_a^-$ also satisfy $k$-descent.
\end{lemma}

\begin{proof}
We prove the statement about $\lambda_a^+$.  The statement about $\lambda_a^-$ holds by an analogous argument.  Suppose that $(x,y) \in \lambda_a^+$ and $(x+1,y) \notin \lambda_a^+$.  By the definition of $k$-descent, either $(x,y) \notin \lambda$, or $(x,y) \in \lambda$ and $C_{y-x-1}(\lambda) < y$.  We first consider the case where $(x,y) \notin \lambda$.  Since $(x,y) \in \lambda_a^+$, this implies that $(x,y) \in D_a$.  Note that $(x-1,y) \in \lambda \cap D_{a+1}$ and $(x,y) \notin \lambda$.  By the definition of $k$-descent, we see that $C_a (\lambda) < y$.  It follows that, if $(x',y') \in \lambda \cap D_{a-1}$ with $y' \geq y$ and $(x',y'+1) \notin \lambda$, then $(x',y'+1)$ is an outside corner, and thus in $\lambda_a^+$.  From this we obtain $C_{a-1} (\lambda_a^+) < y$.

On the other hand, if $(x,y) \in \lambda$, then $C_{y-x-1} (\lambda) < y$.  We may assume that $(x,y) \in D_{a+1}$, because otherwise we have $C_{y-x-1}(\lambda^+_a) \leq C_{y-x-1} (\lambda)$.  Then, since $(x+1,y) \notin \lambda^+_a$, we must have $(x+1,y-1) \notin \lambda$.  Since $\lambda$ satisfies $k$-descent and $(x,y-1) \in \lambda \cap D_a$, we see that $C_{a-1} (\lambda) < y-1$.  Since $C_a (\lambda) < y$ and $C_{a-1} (\lambda) < y-1$, we see that $C_a (\lambda^+_a) < y$.
\end{proof}

We now establish that this is an alternate characterization of $k$-cores.

\begin{proof}[Proof of Proposition~\ref{Prop:Char}]
First, let $\lambda$ be a $k$-core.  By Lemma~\ref{Lem:Transpose}, $\lambda^T$ is a $k$-core.  It therefore suffices to prove that $\lambda$ satisfies $k$-descent.  By definition, $\lambda$ is obtained from the empty partition by a sequence of upward displacements with respect to congruence classes in $\mathbb{Z}/k\mathbb{Z}$.  We prove that $\lambda$ satisfies $k$-descent by induction on the number of upward displacements in this sequence.  The base case is the empty partition, which satisfies $k$-descent trivially.  The inductive step follows from Lemma~\ref{Lem:RowColCondition}, which says that the upward displacement of a partition satisfying $k$-descent also satisfies $k$-descent.

Now, let $\lambda$ be a partition such that both $\lambda$ and $\lambda^T$ satisfy $k$-descent.  We prove that $\lambda$ is a $k$-core by induction on the number of boxes in $\lambda$.  The base case is the empty partition, which is a $k$-core.  If $\lambda$ is non-empty, then there is an inside corner $(x,y) \in \lambda$.  By Lemma~\ref{Lem:RowColCondition}, the downward displacements $\lambda^-_{y-x}$ and $(\lambda^-_{y-x})^T = (\lambda^T)^-_{x-y}$ satisfy $k$-descent.  By induction, $\lambda^-_{y-x}$ is therefore a $k$-core, hence by definition, $(\lambda^-_{y-x})^+_{y-x}$ is a $k$-core as well.  By Lemma~\ref{lem:DisplacementsInvertible}, however, $\lambda = (\lambda^-_{y-x})^+_{y-x}$, so $\lambda$ is a $k$-core.
\end{proof}

\subsection{Behavior of Invariants Under Displacement}

A consequence of this characterization is that $\mathcal{P}_k$ is a graded poset.  To see this, given a vector $\bC = (C_0 , C_1 , \ldots , C_{k-1})$ and a congruence class $a \in \ZZ/k\ZZ$, define the vector $\bC_a^- = (C_{0a}^- , C_{1a}^-, \ldots , C_{k-1a}^-)$ by
\begin{displaymath}
C_{ba}^- = \left\{ \begin{array}{ll}
C_a - 1 & \textrm{if $b = a-1$} \\
C_{a-1} & \textrm{if $b = a$} \\
C_b & \textrm{otherwise.}
\end{array} \right.
\end{displaymath}
The notation is justified by the following proposition.

\begin{proposition}
\label{prop:ColChange}
If $\lambda \in \mathcal{P}_k$ has an inside corner in $D_a$, then $\bC (\lambda_a^-) = \bC (\lambda)_a^-$ .
\end{proposition}

\begin{proof}
It is straightforward to see that, if $b \neq a,a-1$, then $C_b (\lambda_a^-) = C_b (\lambda)$.  By Lemma~\ref{lem:LeftmostCornerA}, the tallest column of $\lambda$ whose last box is in $D_a$ contains an inside corner, and by Lemma~\ref{lem:CongAFurtherLeft}, $C_{a-1}(\lambda) < C_a(\lambda)$.  It follows that $C_{a-1} (\lambda_a^-) = C_a (\lambda) - 1$. 

Now, suppose that $(x,y) \in \lambda \cap D_a$ is the last box of a column.  If $y > C_{a-1} (\lambda)$, then $(x,y)$ is an inside corner of $\lambda$, because $(x+1,y) \in D_{a-1}$ cannot be in $\lambda$ by definition.  It follows that $(x,y) \notin \lambda_a^-$, and thus that $C_a (\lambda_a^-) \leq C_{a-1} (\lambda)$.  We now show that equality holds.  If $C_{a-1} (\lambda) = 0$, then there is nothing to show.  Otherwise, suppose that column $x$ is the tallest column whose last box is in $D_{a-1}$.  By Lemma~\ref{lem:FindColLeft}, $(x,C_{a-1}(\lambda)+1)$ cannot be an outside corner of $\lambda$, hence $x > 1$ and $(x-1,C_{a-1}(\lambda)+1) \notin \lambda$.  It follows that $(x-1,C_{a-1}(\lambda)) \in D_a$ is the last box in its column.  Since $(x-1,C_{a-1}(\lambda))$ is not an inside corner, it is contained in $\lambda_a^-$, so $C_a (\lambda_a^-) \geq C_{a-1} (\lambda)$.  
\end{proof}

\begin{corollary}
\label{cor:PosetIsPoset}
The set $\mathcal{P}_k$ is a graded poset with rank function $\rho_k$.
\end{corollary}

\begin{proof}
Let $\lambda \in \mathcal{P}_k$, and suppose that $\lambda$ has an inside corner in $D_a$.  It suffices to show that 
\[
\rho_k (\lambda) = \rho_k (\lambda_a^-) + 1.
\]
This follows from Proposition~\ref{prop:ColChange} by summing over all $b \in \ZZ/k\ZZ$.
\end{proof}

\subsection{Saturated Tableaux}

Corollary~\ref{cor:PosetIsPoset} provides a natural interpretation for the function $\rho_k$.  As we shall see in Corollary~\ref{cor:kDispMin}, if $\lambda \in \mathcal{P}_k$, then $\rho_k (\lambda)$ is the minimal number of symbols in a $k$-uniform displacement tableau on $\lambda$.  Let $\mathscr{C}(\mathcal{P})$ denote the set of maximal chains in a poset $\mathcal{P}$.  Given a partition $\lambda \in \mathcal{P}_k$, we define a map
\[
\Phi_{\lambda} : {{[g]}\choose{\rho_k (\lambda)}} \times \mathscr{C}(\mathcal{P}_k (\lambda)) \to YT_k (\lambda)
\]
as follows.  Let 
\[
s_1 < s_2 < \cdots < s_{\rho_k (\lambda)}
\]
be the elements of $S \subseteq [g]$, and let
\[
\emptyset = \lambda_0 < \lambda_1 < \cdots < \lambda_{\rho_k (\lambda)} = \lambda
\]
be a maximal chain in $\mathcal{P}_k (\lambda)$.  Define the tableau $t = \Phi_{\lambda} (S, \vec{\lambda})$ by setting 
\[
t(x,y) = s_j \mbox{ if } (x,y) \in \lambda_j \smallsetminus \lambda_{j-1}.
\]
For each $j$, every symbol in $\lambda_{j-1}$ is smaller than $s_j$, so $t$ is a tableau.  Moreover, every box containing the symbol $s_j$ is in the same diagonal (mod $k$), so $t$ satisfies $k$-uniform displacement.  We say that a tableau $t$ on $\lambda$ is \emph{$k$-saturated} if it is in the image of $\Phi_{\lambda}$.  Note that every $k$-saturated tableau contains exactly $\rho_k (\lambda)$ distinct symbols.

\begin{theorem}
\label{Thm:Equidim}
Let $\lambda$ be a $k$-core, and let $t$ be a $k$-uniform displacement tableau on $\lambda$.  Then there exists a $k$-saturated tableau $t'$ on $\lambda$ such that:
\begin{enumerate}
\item  every symbol in $t'$ is a symbol in $t$, and
\item  if $t(x,y) = t'(x',y')$, then $y-x \equiv y'-x' \pmod{k}$.
\end{enumerate} 
\end{theorem}

\begin{proof}
We prove this by induction on $\rho_k (\lambda)$.  The base case is when $\rho_k (\lambda) = 0$, in which case $\lambda$ is the empty partition, and the result is trivial.

For the inductive step, suppose that $h$ is the largest symbol in $t$.  Note that any box containing $h$ must be an inside corner of $\lambda$, and every such box is contained in the same diagonal $D_a$.  In particular, the symbol $h$ does not appear in the restriction $t \vert_{\lambda_a^-}$.  By induction, there exists a $k$-saturated tableau $t''$ on $\lambda_a^-$ such that every symbol in $t''$ is a symbol in $t \vert_{\lambda_a^-}$, and if $t(x,y) = t''(x',y')$, then $y-x \equiv y'-x' \pmod{k}$.

By Corollary~\ref{cor:PosetIsPoset}, $\rho_k (\lambda_a^-) = \rho_k (\lambda) - 1$, so the set $S$ of symbols in $t''$ has size $\rho_k (\lambda) -1$.  By definition, there is a maximal chain
\[
\emptyset = \lambda_0 < \lambda_1 < \cdots < \lambda_{\rho_k (\lambda)-1} = \lambda_a^-
\]
such that $t'' = \Phi_{\lambda_a^-} (S,\vec{\lambda})$.  Let $S' = S \cup \{ h \}$, let $\vec{\lambda}'$ be the chain obtained by appending $\lambda$ to the end of $\vec{\lambda}$, and let $t' = \Phi_{\lambda} (S', \vec{\lambda}')$.   In other words,
\begin{displaymath}
t'(x,y) = \left\{ \begin{array}{ll}
t''(x,y) & \textrm{if $(x,y) \in \lambda_a^-$} \\
h & \textrm{if $(x,y) \notin \lambda_a^-$.}
\end{array} \right.
\end{displaymath}

Clearly, every symbol in $t'$ is a symbol in $t$.  Since $h$ is larger than every symbol appearing in $t \vert_{\lambda_a^-}$, we see that $t'$ is a tableau.  Finally, since every box containing $h$ is in $D_a$, we see that if $t(x,y) = h$, then $y-x \equiv a \pmod{k}$.
\end{proof}

\begin{remark}
\label{Rmk:Dim}
Under the bijection between $k$-uniform displacement tableaux on $k$-cores and words in the affine symmetric group, Theorem~\ref{Thm:Equidim} is equivalent to the statement that every word is equivalent to a reduced word.
\end{remark}

\begin{example}
\label{Ex:Alg}
Given a $k$-uniform displacement tableau $t$ on $\lambda$, the proof of Theorem~\ref{Thm:Equidim} provides an explicit algorithm for producing the $k$-saturated tableau $t'$.  At each step, find the diagonal $D_a$ containing the largest symbol in $t$.  Replace every inside corner in $D_a$ with this symbol, then downward displace with respect to $a$, and iterate the procedure.

Figure~\ref{Fig:Replace} illustrates this procedure for a 3-uniform displacement tableau on $\lambda (\bmu)$, where $\bmu = (-3,-1,1)$.  The tableau on the left uses 8 symbols.  At each step, we highlight in gray the downward displacement of the previous partition in the sequence, replacing symbols as we go until we arrive at a tableau with $\rho_3 (\lambda(\bmu)) = 5$ symbols.

\begin{figure}[h]
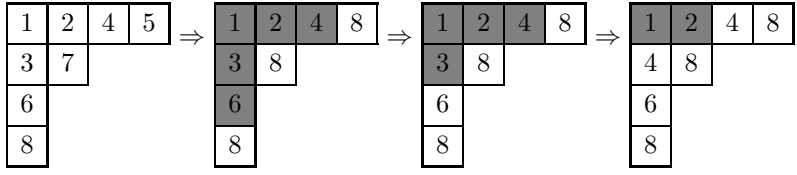

\begin{ytableau}
1 & 2 & 4 & 5\\
3 & 7\\
6\\
8\\
\end{ytableau}
$\Rightarrow$ 
\begin{ytableau}
*(gray)1 & *(gray)2 & *(gray)4 & 8\\
*(gray)3 & 8\\
*(gray)6\\
8\\
\end{ytableau} 
$\Rightarrow$ 
\begin{ytableau}
*(gray)1 & *(gray)2 & *(gray)4 & 8\\
*(gray)3 & 8\\
6\\
8\\
\end{ytableau}
$\Rightarrow$
\begin{ytableau}
*(gray)1 & *(gray)2 & 4 & 8\\
4 & 8\\
6\\
8\\
\end{ytableau}

\caption{Starting with the tableau on the left, we produce a 3-uniform displacement tableau with only 5 symbols.}
\label{Fig:Replace}
\end{figure}

\end{example}

\begin{corollary}
\label{cor:kDispMin}
Let $\lambda$ be a $k$-core.  The minimum number of symbols in a $k$-uniform displacement tableau on $\lambda$ is $\rho_k (\lambda)$.
\end{corollary}

\begin{proof}
Let $t$ be a $k$-uniform displacement tableau on $\lambda$.  By Theorem~\ref{Thm:Equidim}, there exists a $k$-uniform displacement tableau $t'$ on $\lambda$ such that every symbol in $t'$ is a symbol in $t$, and $t'$ has exactly $\rho_k (\lambda)$ symbols.  It follows that $t$ has at least $\rho_k (\lambda)$ symbols.
\end{proof}

\section{Dimensions of Tropical Splitting Type Loci}
\label{Sec:Dim}

In this section, we compute the dimension of $\overline{W}^{\bmu} (\Gamma)$, proving Theorem~\ref{thm:TropEquiDim}.  In order to do this, we first apply the results of Section~\ref{Sec:Partitions} to $k$-staircases.

\begin{lemma}
Let $\bmu \in \ZZ^k$ be a splitting type, and let $c(\bmu) = - \sum_{i=1}^k \mu_i$.  Then every inside corner of $\lambda (\bmu)$ is in $D_{c(\bmu)}$.
\end{lemma}

\begin{proof}
Recall that the inside corners of $\lambda (\bmu)$ are the boxes $(x_m (\bmu), y_m (\bmu))$.  By definition, we have
\begin{align*}
y_m (\bmu) - x_m (\bmu) &= \sum_{i=1}^k \Big( \max \{ 0, -\mu_i - m - 1 \} - \max \{0,  \mu_i + m + 1 \} \Big) \\
&= \sum_{i=1}^k \Big( \max \{ 0, -\mu_i - m - 1 \} + \min \{0, - \mu_i - m - 1 \} \Big) &= \sum_{i=1}^k ( -\mu_i - m - 1 ) \\
& &\equiv - \sum_{i=1}^k \mu_i \pmod{k}  .
\end{align*}
\end{proof}

If $\lambda$ is a $k$-staircase, then there is a simple expression for the invariants $C_a (\lambda)$.

\begin{lemma}
\label{lem:SplitTypeCol}
Let $\bmu \in \ZZ^k$ be a splitting type.  Then
\[
C_{c(\bmu)+i} (\lambda (\bmu)) = y_{-\mu_{k-i}} (\bmu) = \sum_{j=1}^{k-1-i} \max \{ 0, \mu_{k-i} - \mu_j - 1 \} \mbox{ for all  } 0 \leq i \leq k-1 .
\]
\end{lemma}

\begin{proof}
We first identify the congruence classes $a \in \ZZ/k\ZZ$ such that $C_a (\lambda (\bmu)) = 0$.  Let $(x,y)$ be the last box in a column of $\lambda (\bmu)$.  Then there exists an integer $m$ such that $y = y_m (\bmu)$ and $x_{m-1} (\bmu) < x \leq x_m (\bmu)$.  Since $(x_m (\bmu), y_m (\bmu)) \in D_{c(\bmu)}$, we see that $(x,y) \in D_{c(\bmu) + i}$ for some $i$ in the range $0 \leq i < \alpha_m (\bmu)$.  Since $\alpha_m (\bmu) \leq \alpha_{m+1} (\bmu)$ for all $m$, we may reduce to the case where $m = -2 - \mu_1$ is maximal.  We see that $C_{c(\bmu)+i} (\lambda (\bmu))$ is nonzero for $i$ in the range $0 \leq i < \alpha_{-2-\mu_1}(\bmu)$ and zero for $i$ in the range $\alpha_{-2-\mu_1}(\bmu) \leq i \leq k-1$.  Note that $\alpha_{-2-\mu_1}(\bmu)$ is the minimal index $j$ such that $\mu_{j+1} \geq \mu_1 + 2$.

To establish the formula when $C_a (\lambda (\bmu))$ is nonzero, we proceed by induction on the number of rows of $\lambda(\bmu)$. The base case is when $\mu_j - \mu_i \leq 1$ for all $i<j$, in which case $\lambda(\bmu)$ is the empty partition.  In this case, $C_i (\lambda (\bmu)) = y_{-\mu_{k-i}} (\bmu) = 0$ for all $0 \leq i \leq k-1$.

For the inductive step, recall from Lemma~\ref{Lem:DeleteRow} that $\lambda(\bmu^+)$ is the partition obtained by deleting the first row of $\lambda(\bmu)$.  It follows that
\begin{displaymath}
C_{a+1} (\lambda(\bmu^+)) = \left\{ \begin{array}{ll}
C_a (\lambda(\bmu)) - 1 & \textrm{if $C_a (\lambda(\bmu)) \neq 0$} \\
0 & \textrm{if $C_a (\lambda(\bmu)) = 0$.}
\end{array} \right.
\end{displaymath}
Note that $c (\bmu^+) = c(\bmu) + 1$.  If $\mu_{k-i} \leq \mu_1 + 1$, then $C_{c(\bmu^+)+i} (\lambda(\bmu^+)) = y_{-\mu^+_{k-i}} (\bmu^+) = 0$.  By induction, if $\mu_{k-i} \geq \mu_1 + 2$, then
\[
C_{c(\bmu^+)+i} (\lambda(\bmu^+)) = \sum_{j=1}^{k-1-i} \max \{ 0, \mu_{k-i} - \mu^+_j - 1 \} = \sum_{j=1}^{k-1-i} \max \{ 0, \mu_{k-i} - \mu_j - 1 \} - 1,
\]
and the result follows.
\end{proof}

\begin{corollary}
\label{cor:SplitTypeMin}
Let $\bmu \in \ZZ^k$ be a splitting type.  Then $\rho_k (\lambda (\bmu)) = \vert \bmu \vert$.
\end{corollary}

\begin{proof}
By Lemma~\ref{lem:SplitTypeCol}, we have
\begin{align*}
\rho_k (\lambda (\bmu)) &= \sum_{i=0}^{k-1} C_{c(\bmu)+i} (\lambda (\bmu)) \\
&= \sum_{i=0}^{k-1} \sum_{j=1}^{k-1-i} \max \{ 0, \mu_{k-i} - \mu_j - 1 \} \\
&= \sum_{j<i} \max \{ 0, \mu_i - \mu_j -1 \} = \vert \bmu \vert.
\end{align*}
\end{proof}

In order to use the results of Section~\ref{Sec:Partitions}, we must show that $k$-staircases are in $\mathcal{P}_k$.

\begin{proposition}
\label{prop:SplitTypeiskDisp}
Every $k$-staircase is a $k$-core.
\end{proposition}

\begin{proof}
Let $\bmu \in \ZZ^k$ be a splitting type.  By Proposition~\ref{Prop:Char}, we must show that $\lambda (\bmu)$ and $\lambda(\bmu)^T$ satisfy $k$-descent.  By Lemma~\ref{lem:SerreDualPartition}, it suffices to show that $\lambda (\bmu)$ satisfies $k$-descent.  Let $(x,y) \in \lambda(\bmu) \cap D_a$ and suppose that $(x+1,y) \notin \lambda(\bmu)$.  We will show that $C_{a-1} (\lambda (\bmu)) < y$.  By assumption, there is an integer $m$ such that $x = x_m (\bmu)$ and $y_{m+1} (\bmu) < y \leq y_m (\bmu)$.  Since $(x_{m+1} (\bmu), y_{m+1} (\bmu)) \in D_{c(\bmu)}$, we see that $(x,y) \in D_{c(\bmu) + i}$ for some $i$ in the range $\alpha_{m+1} (\bmu) < i \leq k$.  By Lemma~\ref{lem:SplitTypeCol}, we have
\[
C_{c(\bmu)-i-1} (\lambda (\bmu)) = y_{-\mu_{k-i+1}} (\bmu) . 
\]
If $m+1 \geq -\mu_{k-i+1} (\bmu)$, then $\alpha_{m+1} (\bmu) \geq i$, a contradiction.  It follows that 
\[
y_{-\mu_{k-i+1}} (\bmu) < y_{m+1} (\bmu) < y.
\]
\end{proof}

We now prove the main theorem.

\begin{theorem}
\label{thm:UnionOverPhi}
Let $\Gamma$ be a $k$-gonal chain of loops of genus $g$, and let $\bmu \in \ZZ^k$ be a splitting type.  Then
\[
\overline{W}^{\bmu} (\Gamma) = \bigcup \TT (t) ,
\]
where the union is over all $k$-saturated tableaux on $\lambda (\bmu)$ with alphabet $[g]$.
\end{theorem}

\begin{proof}
Let $t$ be a $k$-uniform displacement tableau on $\lambda (\bmu)$.  By Theorem~\ref{thm:TropClassify}, it suffices to show that there is a $k$-saturated tableau $t'$ on $\lambda (\bmu)$ such that $\TT (t) \subseteq \TT (t')$.  By Proposition~\ref{prop:SplitTypeiskDisp}, $\lambda (\bmu)$ is a $k$-core.  Thus, by Theorem~\ref{Thm:Equidim}, there is a $k$-saturated tableau $t'$ such that every symbol in $t'$ is a symbol in $t$ and, if $t(x,y) = t(x',y')$, then $y-x \equiv y'-x' \pmod{k}$.  By Lemma~\ref{lem:contain}, we have $\TT (t) \subseteq \TT (t')$.
\end{proof}

\begin{proof}[Proof of Theorem~\ref{thm:TropEquiDim}]
Recall that the codimension of $\TT (t)$ is equal to the number of symbols in $t$.  The result then follows from Theorem~\ref{thm:UnionOverPhi} because every $k$-saturated tableau on $\lambda$ contains exactly $\rho_k (\lambda)$ symbols, and by Corollary~\ref{cor:SplitTypeMin}, $\rho_k (\lambda(\bmu)) = \vert \bmu \vert$. 
\end{proof}

We now explain the connection between the tropical geometry and classical algebraic geometry.  The following has become a standard argument in tropical geometry, for instance in \cite{CDPR, PfluegerkGonal, JensenRanganthan, CPJ}.  Recall that, if $\overline{W}^{\bmu} (C)$ is nonempty, then $\dim \overline{W}^{\bmu} (C) \geq g - \vert \bmu \vert$.  We show the reverse inequality.

\begin{proof}[Proof of Theorem~\ref{Thm:MainThm}]
By \cite[Lemma~2.4]{PfluegerkGonal}, there exists a curve $C$ of genus $g$ and gonality $k$ over a nonarchimedean field $K$ with skeleton $\Gamma$.  By Proposition~\ref{prop:specialization}, we have
\[
\Trop \left( \overline{W}^{\bmu} (C) \right) \subseteq \overline{W}^{\bmu} (\Gamma) .
\]
By \cite[Theorem~6.9]{Gubler07}, we have
\[
\dim \overline{W}^{\bmu} (C) = \dim \Trop \left( \overline{W}^{\bmu} (C) \right) \leq \dim \overline{W}^{\bmu} (\Gamma) = g - \vert \bmu \vert ,
\]
where the last equality comes from Theorem~\ref{thm:TropEquiDim}.
\end{proof}

\section{Connectedness of Tropical Splitting Type Loci}
\label{Sec:Connect}

In this section, we prove Theorem~\ref{thm:connect}, which says that $\overline{W}^{\bmu} (\Gamma)$ is connected in codimension one.  We borrow the ideas and terminology from \cite[Section~4.2]{CLRW}.

Let $t$ be a $k$-uniform displacement tableau, let $a$ be a symbol that is not in $t$, and let $b$ be either the smallest symbol in $t$ that is greater than $a$ or the largest symbol in $t$ that is smaller than $a$.  If we take a proper subset of the boxes containing $b$ and replace them with $a$, then we obtain a $k$-uniform displacement tableau $t'$, with $\TT(t') \subset \TT(t)$ and $\dim \TT(t') = \dim \TT(t) -1$.  If we instead replace every instance of the symbol $b$ in $t$ with the symbol $a$, then we obtain a $k$-uniform displacement tableau $t'$, with $\dim \TT(t') = \dim \TT(t)$, such that $\TT(t)$ and $\TT(t')$ intersect in codimension one.  This procedure is called \emph{swapping} in $a$ for $b$.

Given a symbol $b$ in $t$, we obtain a $k$-uniform displacement tableau $t'$ without the symbol $b$, by iterating the procedure above.  If there is a symbol $a < b$ that is not in $t$, then the resulting tableau can be described explicitly:
\begin{displaymath}
t'(x,y) = \left\{ \begin{array}{ll}
t(x,y)-1 & \textrm{if $a < t(x,y) \leq b$} \\
t(x,y) & \textrm{otherwise.}
\end{array} \right.
\end{displaymath}
If there is a symbol $a > b$ that is not in $t$, then $t'$ is obtained instead by increasing by 1 every symbol in $t$ between $b$ and $a$.  Because $t'$ is obtained by a sequence of swaps, we see that there is a chain of tori from $\TT(t)$ to $\TT(t')$, such that each consecutive pair of tori in the chain intersect in codimension one.  This procedure is called \emph{cycling} out $b$.

\begin{proof}[Proof of Theorem~\ref{thm:connect}]
Let $t, t'$ be $k$-saturated tableaux on $\lambda (\bmu)$.  By Theorem~\ref{thm:UnionOverPhi}, it suffices to construct a sequence
\[
t = t_0 , t_1 , \ldots , t_m = t'
\]
of $k$-saturated tableaux, where $\TT (t_i)$ and $\TT(t_{i+1})$ intersect in codimension one for all $i$.  Both $t$ and $t'$ contain precisely $\vert \bmu \vert$ symbols.  By cycling out all symbols greater than $\vert \bmu \vert$, we may assume that the symbols in $t$ and $t'$ are precisely those in $[ \vert \bmu \vert ]$.  In other words, there exist maximal chains
\begin{align*}
\emptyset &= \lambda_0 < \lambda_1 < \cdots < \lambda_{\vert \bmu \vert} = \lambda(\bmu),\\
\emptyset &= \lambda'_0 < \lambda'_1 < \cdots < \lambda'_{\vert \bmu \vert} = \lambda(\bmu)
\end{align*}
such that $t = \Phi([ \vert \bmu \vert ], \vec{\lambda})$ and $t' = \Phi([ \vert \bmu \vert ], \vec{\lambda}')$.  If $\vec{\lambda}$ and $\vec{\lambda}'$ coincide, then $t = t'$, and we are done.

We prove the remaining cases by induction, having just completed the base case.  Let $j$ be the largest symbol such that $\lambda_{j-1} \neq \lambda'_{j-1}$.  Equivalently, the symbols $j+1, \ldots , \vert \bmu \vert$ appear in the same set of boxes of $t$ and $t'$.  We will construct a sequence
\[
t = t'_0 , t'_1 , \ldots , t'_n = t''
\]
of $k$-saturated tableaux, where $\TT (t'_i)$ and $\TT(t'_{i+1})$ intersect in codimension one for all $i$, and where each of the symbols $j, \ldots , \vert \bmu \vert$ appears in the same set of boxes of $t'$ and $t''$.

Since $g > \vert \bmu \vert$, either $g = j+1$ or there exists a symbol in $[g]$ that is greater than $j+1$.  We let $\widehat{t}$ be the tableau obtained by cycling $j+1$ out of $t$.  In other words,
\begin{displaymath}
\widehat{t}(x,y) = \left\{ \begin{array}{ll}
t(x,y) & \textrm{if $t(x,y) \leq j$} \\
t(x,y)+1 & \textrm{if $t(x,y) > j$.}
\end{array} \right.
\end{displaymath}
We define
\begin{displaymath}
\widetilde{t}(x,y) = \left\{ \begin{array}{ll}
j+1 & \textrm{if $(x,y) \in \lambda'_j \smallsetminus \lambda'_{j-1}$} \\
\widehat{t}(x,y) & \textrm{otherwise.} \\
\end{array} \right.
\end{displaymath}
To see that $\widetilde{t}$ is a tableau, note that
\[
\lambda_j = \lambda'_j = \{ (x,y) \in \lambda (\bmu) \mbox{ } \vert \mbox { } \widehat{t} (x,y) \leq j \},
\]
and $\widehat{t}$ does not contain the symbol $j+1$, so every box in $\lambda (\bmu) \smallsetminus \lambda'_j$ contains a symbol that is greater than $j+1$, and every box in $\lambda'_{j-1}$ contains a symbol that is smaller than $j+1$.  Note that $\widetilde{t}$ contains one more symbol than $\widehat{t}$, so $\TT (\widetilde{t}) \subset \TT (\widehat{t})$ has codimension 1.  Applying the procedure of Example~\ref{Ex:Alg}, we obtain a $k$-saturated tableau $\widetilde{t}'$ such that $\TT (\widetilde{t}) \subset \TT (\widetilde{t'})$.  Since $i+1$ is the largest symbol in $\lambda'_i$ for all $i \geq j$, we see that $\widetilde{t}' (x,y) = \widetilde{t} (x,y)$ for all $(x,y) \in \lambda (\bmu) \smallsetminus \lambda'_{j-1}$.  Finally, we let $t''$ be the tableau obtained by cycling out all symbols greater than $\vert \bmu \vert$ from $\widetilde{t}'$.  By construction, each of the symbols $j, \ldots , \vert \bmu \vert$ appears in the same set of boxes of $t'$ and $t''$.
\end{proof}

\begin{remark}
\label{Rmk:Connect}
Under the bijection with words in the affine symmetric group, Theorem~\ref{thm:connect} is equivalent to the statement that any two reduced expressions for the same word can be connected via a sequence of ``braid moves'' (see \cite[Theorem~3.3.1]{BjornerBrenti}).
\end{remark}

\begin{example}
Figure~\ref{Fig:Connect} illustrates the procedure in the proof of Theorem~\ref{thm:connect}.  The two tableaux $t, t'$ on the ends correspond to two maximal-dimension tori in $\overline{W}^{\bmu} (\Gamma)$, where $\bmu = (-3,-1,1)$.  If $g \geq 6$, we construct a chain of tori from $\TT(t)$ to $\TT(t')$ in this tropical splitting type locus, where each torus intersects the preceding torus in codimension one.  The largest symbol where $t$ and $t'$ disagree is 4.  We therefore begin by cycling out 5, to obtain the second tableau in the chain.  We then place a 5 in each box where a 4 appears in $t'$, to obtain the third tableau in the chain, using all 6 symbols.  Applying the procedure of Example~\ref{Ex:Alg}, we obtain the fourth tableau.  Finally, by cycling out 6, we arrive at $t'$.

\begin{figure}[h]
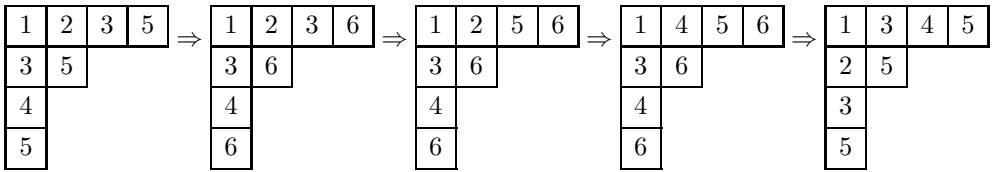

$\begin{ytableau}
1 & 2 & 3 & 5\\
3 & 5\\
4\\
5\\
\end{ytableau}\Rightarrow
\begin{ytableau}
1 & 2 & 3 & 6\\
3 & 6\\
4\\
6\\
\end{ytableau}\Rightarrow
\begin{ytableau}
1 & 2 & 5 & 6\\
3 & 6\\
4\\
6\\
\end{ytableau}\Rightarrow
\begin{ytableau}
1 & 4 & 5 & 6\\
3 & 6\\
4\\
6\\
\end{ytableau}\Rightarrow\begin{ytableau}
1 & 3 & 4 & 5\\
2 & 5\\
3\\
5\\
\end{ytableau}$
\caption{If $g \geq 6$, then $\overline{W}^{(-3,-1,1)} (\Gamma)$ is connected in codimension 1.}
\label{Fig:Connect}
\end{figure}

\end{example}

\section{Cardinality of Tropical Splitting Type Loci}
\label{Sec:Count}

We begin this section by proving Theorem~\ref{thm:count}.

\begin{proof}[Proof of Theorem~\ref{thm:count}]
By Theorem~\ref{thm:UnionOverPhi},
\[
\overline{W}^{\bmu} (\Gamma) = \bigcup \TT (t) ,
\]
where the union is over all $k$-saturated tableaux on $\lambda (\bmu)$ with alphabet $[g]$.  Since $g = \vert \bmu \vert$, each torus $\TT(t)$ in this union is 0-dimensional, and therefore consists of a single divisor class.  Consider the composition of $\Phi_{\lambda (\bmu)}$ with the map sending a tableau $t$ to the unique divisor class in $\TT(t)$.  By the above, this composition surjects onto $\overline{W}^{\bmu} (\Gamma)$, and it suffices to show that it is injective.  Let
\[
\emptyset = \lambda_0 < \lambda_1 < \cdots < \lambda_g = \lambda (\bmu)
\]
\[
\emptyset = \lambda'_0 < \lambda'_1 < \cdots < \lambda'_g = \lambda (\bmu)
\]
be distinct maximal chains in $\mathcal{P}_k (\lambda (\bmu))$, and let $j$ be the minimal index such that $\lambda'_j \neq \lambda_j$.  By definition, $\lambda_j = \lambda^+_{j-1,a}$ and $\lambda'_j = \lambda^+_{j-1,b}$ for some $a \not\equiv b \pmod{k}$.  It follows that, if $\TT(t) = \{ D \}$, then $\xi_j (D) \equiv a \not\equiv b \pmod{k}$, so $D \notin \TT(t')$.  Therefore, every maximal chain in $\mathcal{P}_k$ corresponds to a distinct divisor class in $\overline{W}^{\bmu} (\Gamma)$. 
\end{proof}

\subsection{Algorithm for Computing Maximal Chains}

The number of maximal chains in $\mathcal{P}_k (\lambda)$ is an important invariant of a partition $\lambda \in \mathcal{P}_k$, not only because of Theorem~\ref{thm:count}, but also because of its connection to the affine symmetric group \cite{LapointeMorse}.  We would therefore like to compute this invariant in examples.  In order to simplify our arguments, we first show that a partition $\lambda \in \mathcal{P}_k$ is uniquely determined by the vector $\bC (\lambda)$.

\begin{lemma}
\label{lem:Reorder}
Let $\lambda, \lambda' \in \mathcal{P}_k$.  If there exists a permutation $\sigma \in S_k$ such that $C_a (\lambda) = C_{\sigma(a)} (\lambda')$ for all $a \in \ZZ/k\ZZ$, then $\lambda = \lambda'$.
\end{lemma}

\begin{proof}
We prove this by induction on $\rho_k (\lambda) = \rho_k (\lambda')$.  The base case is when $\rho_k (\lambda) = 0$, in which case $\lambda = \lambda'$ is the empty partition.  For the inductive step, let
\[
y = \max_{a \in \ZZ/k\ZZ} C_a (\lambda) = \max_{a \in \ZZ/k\ZZ} C_a (\lambda'),
\]
and let $x$ be the number of congruence classes $a \in \ZZ/k\ZZ$ such that $C_a (\lambda) = y$.  By definition, the first $x$ columns of both $\lambda$ and $\lambda'$ must all have height $y$.  If $\lambda$ is nonempty then it has an inside corner.  This implies that $x \leq k-1$ by Lemma~\ref{lem:CongAFurtherLeft}.  It follows that column $x+1$ of both $\lambda$ and $\lambda'$ has height less than $y$, so $(x,y)$ is an inside corner of both partitions, and $y = C_{y-x} (\lambda) = C_{y-x} (\lambda')$.  By Proposition~\ref{prop:ColChange}, there exists a permutation $\pi \in S_k$ such that
\[
C_a (\lambda_{y-x}^-) = C_{\pi(a)} (\lambda'^-_{y-x}) \mbox{ for all } a \in \ZZ/k\ZZ .
\]
By Lemma~\ref{Lem:RowColCondition}, $\lambda_{y-x}^- , \lambda'^-_{y-x} \in \mathcal{P}_k$, hence by induction, $\lambda_{y-x}^- = \lambda'^-_{y-x}$.  Finally, by Lemma~\ref{lem:DisplacementsInvertible}, we have
\[
\lambda = (\lambda_{y-x}^-)_{y-x}^+ = (\lambda'^-_{y-x})_{y-x}^+ = \lambda' .
\]
\end{proof}

Lemma~\ref{lem:Reorder} allows us to simplify arguments by focusing on the vectors $\bC (\lambda)$, rather than the partitions $\lambda$.  For example, Figure~\ref{Fig:Hasse} depicts the Hasse diagram of a principal order ideal in $\mathcal{P}_6$, where each partition $\lambda$ is represented by the vector $\bC (\lambda)$.  

\begin{figure}[h]
\begin{tikzpicture}

\node (0) at (0,0) {$(0,0,0,0,0,0)$};
\node[shape=circle,draw,inner sep=2pt] at (-1.5,0) {1};
\node (1) at (0,1) {$(1,0,0,0,0,0)$};
\node[shape=circle,draw,inner sep=2pt] at (-1.5,1) {1};
\draw (0)--(1);
\node (2a) at (-2,2) {$(0,2,0,0,0,0)$};
\node[shape=circle,draw,inner sep=2pt] at (-3.5,2) {1};
\draw (1)--(2a);
\node (2b) at (2,2) {$(1,0,0,0,0,1)$};
\node[shape=circle,draw,inner sep=2pt] at (0.5,2) {1};
\draw (1)--(2b);
\node (3a) at (-4,3) {$(0,0,3,0,0,0)$};
\node[shape=circle,draw,inner sep=2pt] at (-5.5,3) {1};
\draw (2a)--(3a);
\node (3b) at (0,3) {$(0,2,0,0,0,1)$};
\node[shape=circle,draw,inner sep=2pt] at (-1.5,3) {2};
\draw (2a)--(3b);
\draw (2b)--(3b);
\node (3c) at (4,3) {$(1,0,0,0,1,1)$};
\node[shape=circle,draw,inner sep=2pt] at (2.5,3) {1};
\draw (2b)--(3c);
\node (4a) at (-4,4) {$(0,0,3,0,0,1)$};
\node[shape=circle,draw,inner sep=2pt] at (-5.5,4) {3};
\draw (3a)--(4a);
\draw (3b)--(4a);
\node (4b) at (0,4) {$(2,2,0,0,0,0)$};
\node[shape=circle,draw,inner sep=2pt] at (-1.5,4) {2};
\draw (3b)--(4b);
\node (4c) at (4,4) {$(0,2,0,0,1,1)$};
\node[shape=circle,draw,inner sep=2pt] at (2.5,4) {3};
\draw (3b)--(4c);
\draw (3c)--(4c);
\node (5a) at (-4,5) {$(2,0,3,0,0,0)$};
\node[shape=circle,draw,inner sep=2pt] at (-5.5,5) {5};
\draw (4a)--(5a);
\draw (4b)--(5a);
\node (5b) at (0,5) {$(0,0,3,0,1,1)$};
\node[shape=circle,draw,inner sep=2pt] at (-1.5,5) {6};
\draw (4a)--(5b);
\draw (4c)--(5b);
\node (5c) at (4,5) {$(2,2,0,0,1,0)$};
\node[shape=circle,draw,inner sep=2pt] at (2.5,5) {5};
\draw (4b)--(5c);
\draw (4c)--(5c);
\node (6a) at (-4,6) {$(0,3,3,0,0,0)$};
\node[shape=circle,draw,inner sep=2pt] at (-5.5,6) {5};
\draw (5a)--(6a);
\node (6b) at (0,6) {$(0,0,0,4,1,1)$};
\node[shape=circle,draw,inner sep=2pt] at (-1.5,6) {6};
\draw (5b)--(6b);
\node (6c) at (4,6) {$(2,2,0,0,0,2)$};
\node[shape=circle,draw,inner sep=2pt] at (2.5,6) {5};
\draw (5c)--(6c);
\node (6d) at (8,6) {$(2,0,3,0,1,0)$};
\node[shape=circle,draw,inner sep=2pt] at (9.5,6) {16};
\draw (5a)--(6d);
\draw (5b)--(6d);
\draw (5c)--(6d);
\node (7a) at (-4,7) {$(0,3,3,0,1,0)$};
\node[shape=circle,draw,inner sep=2pt] at (-5.5,7) {21};
\draw (6a)--(7a);
\draw (6d)--(7a);
\node (7b) at (0,7) {$(2,0,0,4,1,0)$};
\node[shape=circle,draw,inner sep=2pt] at (-1.5,7) {22};
\draw (6b)--(7b);
\draw (6d)--(7b);
\node (7c) at (4,7) {$(2,0,3,0,0,2)$};
\node[shape=circle,draw,inner sep=2pt] at (2.5,7) {21};
\draw (6c)--(7c);
\draw (6d)--(7c);
\node (8a) at (-4,8) {$(0,3,0,4,1,0)$};
\node[shape=circle,draw,inner sep=2pt] at (-5.5,8) {43};
\draw (7a)--(8a);
\draw (7b)--(8a);
\node (8b) at (0,8) {$(0,3,3,0,0,2)$};
\node[shape=circle,draw,inner sep=2pt] at (-1.5,8) {42};
\draw (7a)--(8b);
\draw (7c)--(8b);
\node (8c) at (4,8) {$(2,0,0,4,0,2)$};
\node[shape=circle,draw,inner sep=2pt] at (2.5,8) {43};
\draw (7b)--(8c);
\draw (7c)--(8c);
\node (9a) at (-4,9) {$(0,0,4,4,1,0)$};
\node[shape=circle,draw,inner sep=2pt] at (-5.5,9) {43};
\draw (8a)--(9a);
\node (9c) at (4,9) {$(2,0,0,0,5,2)$};
\node[shape=circle,draw,inner sep=2pt] at (2.5,9) {43};
\draw (8c)--(9c);
\node (9d) at (0,9) {$(0,3,0,4,0,2)$};
\node[shape=circle,draw,inner sep=2pt] at (-1.5,9) {128};
\draw (8a)--(9d);
\draw (8b)--(9d);
\draw (8c)--(9d);
\node (10a) at (-2,10) {$(0,0,4,4,0,2)$};
\node[shape=circle,draw,inner sep=2pt] at (-3.5,10) {171};
\draw (9a)--(10a);
\draw (9d)--(10a);
\node (10b) at (2,10) {$(0,3,0,0,5,2)$};
\node[shape=circle,draw,inner sep=2pt] at (0.5,10) {171};
\draw (9c)--(10b);
\draw (9d)--(10b);
\node (11) at (0,11) {$(0,0,4,0,5,2)$};
\node[shape=circle,draw,inner sep=2pt] at (-1.5,11) {342};
\draw (10a)--(11);
\draw (10b)--(11);
\node (12) at (0,12) {$(0,0,0,5,5,2)$};
\node[shape=circle,draw,inner sep=2pt] at (-1.5,12) {342};
\draw (11)--(12);

\end{tikzpicture}
\caption{A principal order ideal in $\mathcal{P}_6$.  The circled values indicate the number of maximal chains below each vector.}
\label{Fig:Hasse}

\end{figure}
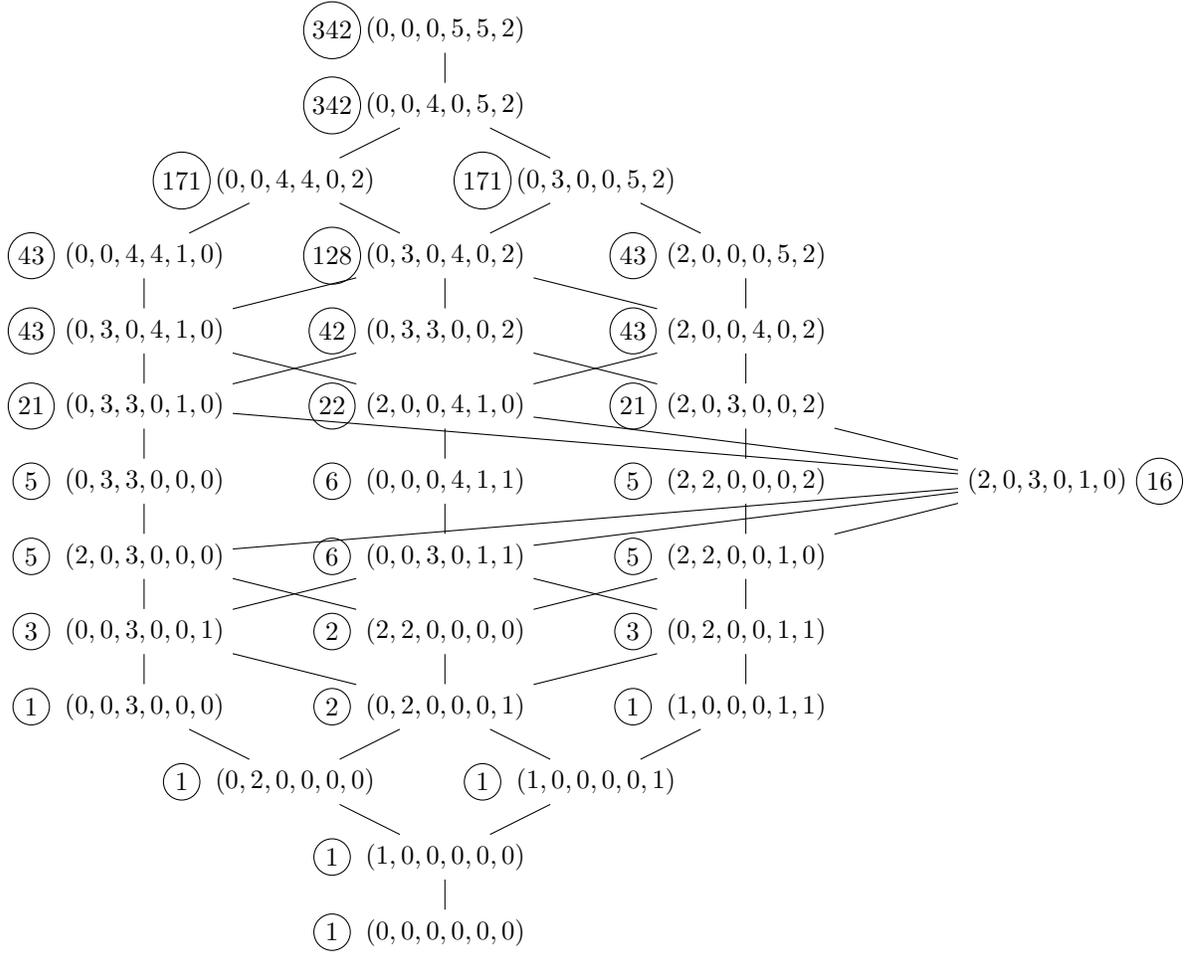

Given a partition $\lambda \in \mathcal{P}_k$, we provide an algorithm for producing the Hasse diagram $\mathcal{P}_k (\lambda)$, as in Figure~\ref{Fig:Hasse}.

\begin{algorithm}
\label{Alg:Recur}
\textbf{Step 1:}  Initialize with the vector $\bC (\lambda)$.

\textbf{Step 2:}  For each vector $\bC$, write below it the vectors $\bC^-_a$, for all $a$ such that $C_{a-1} < C_a$.

\textbf{Step 3:}  Iterate Step 2 for each vector that is written down, until exhaustion.
\end{algorithm}

By Lemma~\ref{lem:Reorder}, the number of partitions in $\mathcal{P}_k$ or rank $\rho$ is less than or equal to the number of partitions of $\rho$ with at most $k-1$ parts.  (In fact, these numbers are equal, see \cite[Proposition~1.3]{LLMSSZ}.)  Together with the fact that each partition covers at most $k-1$ others, this implies that the algorithm terminates in polynomial time for fixed $k$.

We introduce notation that will simplify our examples.  Given $\lambda \in \mathcal{P}_k$, we define $\alpha (\bC (\lambda))$ to be the number of maximal chains in $\mathcal{P}_k (\lambda)$.  By Lemma~\ref{lem:Reorder}, this is well-defined.  We further define $\alpha$ up to cyclic permutation; that is,
\[
\alpha \Big( C_i (\lambda), C_{i+1} (\lambda) , \ldots , C_{i-1} (\lambda) \Big) = \alpha \Big( C_0 (\lambda) , C_1 (\lambda) , \ldots , C_{k-1} (\lambda) \Big) .
\]
Again, by Lemma~\ref{lem:Reorder}, $\alpha$ is well-defined.  Indeed, by Lemma~\ref{lem:Reorder}, $\alpha$ could be defined up to arbitrary permutation, but in practice it is important to keep track of which values $C_a$ are consecutive.  This is because $\alpha$ satisfies the following recurrence.

\begin{lemma}
\label{lem:recur}
For any $\lambda \in \mathcal{P}_k$, we have
\[
\alpha (\bC (\lambda)) = \sum_{\substack{a \in \ZZ/ k\ZZ \text{ s.t.} \\ C_{a-1} (\lambda) < C_a (\lambda)}} \alpha (\bC (\lambda)^-_a ) .
\]
\end{lemma}

\begin{proof}
The number of maximal chains in $\mathcal{P}_k (\lambda)$ is equal to the sum, over $\lambda' \in \mathcal{P}_k$ covered by $\lambda$, of the number of maximal chains in $\mathcal{P}_k (\lambda')$.  By definition, $\lambda' \in \mathcal{P}_k$ is covered by $\lambda$ if and only if $\lambda' = \lambda^-_a$ and $\lambda$ has an inside corner in $D_a$.  By Lemma~\ref{lem:CongAFurtherLeft}, $\lambda$ has an inside corner in $D_a$ if and only if $C_{a-1} (\lambda) < C_a (\lambda)$.  The result then follows from Proposition~\ref{prop:ColChange}.
\end{proof}

Using Algorithm~\ref{Alg:Recur} and Lemma~\ref{lem:recur}, one can compute $\alpha (\bC (\lambda))$ recursively.  Start at the bottom of the Hasse diagram, note that $\alpha (\vec{0}) = 1$, and then proceed upwards, summing the numbers that appear directly below each vector.  These numbers appear in the circles in Figure~\ref{Fig:Hasse}.

\subsection{Examples}

The remainder of the paper consists of examples, using Lemma~\ref{lem:recur} to compute the number of maximal chains in $\mathcal{P}_k (\lambda (\bmu))$ for various splitting types $\bmu$.  In many cases, we will see that this number agrees with the cardinality of $\overline{W}^{\bmu} (C)$ for general $(C,\pi) \in \mathcal{H}_{g,k}$.  In each case, we assume that $g = \vert \bmu \vert$.  By Theorem~\ref{Thm:MainThm}, this implies that $W^{\bmu} (C) = \overline{W}^{\bmu} (C)$.

\begin{example}
\label{Ex:Classic}
If $-2 \leq \mu_j \leq 0$ for all $j$, then $\lambda (\bmu) = \lambda_0 (\bmu)$ is a rectangle, and every $k$-uniform displacement tableau on $\lambda (\bmu)$ is a standard Young tableau.  The number of such tableaux is counted by the standard hook-length formula:
\[
\vert \overline{W}^{\bmu} (\Gamma) \vert = \vert \bmu \vert ! \prod_{j = 0}^{x_0 (\bmu)-1} \frac{j!}{(y_0 (\bmu) + j)!} .
\]
It is a classical result, due to Castelnuovo, that this formula also yields the number of $g^r_d$'s on a general curve of genus $\vert \bmu \vert$, where $r = x_0 (\bmu) - 1$, and $d = d (\bmu)$ \cite[p.211]{ACGH}.
\end{example}

\begin{example}
\label{Ex:OneCol}
If $\mu_j$ is equal to either $\mu_1$ or $\mu_1 + 1$ for each $j < k$, then $d(\bmu) = k\mu_k$ and up to cyclic permutation we have $\bC (\lambda (\bmu)) = (\vert \bmu \vert , 0, 0, \ldots , 0)$.  For ease of notation, we write this as $\bC (\lambda (\bmu)) = (\vert \bmu \vert , 0^{(k-1)})$.  We show that $\alpha (z,0^{(k-1)}) = 1$.  This is easy to see by induction on $z$.  It is clear that $\alpha (1,0^{(k-1)}) = 1$, and by Lemma~\ref{lem:recur}, we have
\[
\alpha (z,0^{(k-1)}) = \alpha (0^{(k-1)},z-1) = \alpha (z-1,0^{(k-1)}) .
\]

Now, if $D \in \overline{W}^{\bmu} (C)$, then by definition, $\deg D = k\mu_k$ and $D - \mu_k g^1_k$ is effective.  It follows that $\overline{W}^{\bmu} (C) = \{ \mu_k g^1_k \}$.  This splitting type locus therefore has cardinality 1, equal to that of $\overline{W}^{\bmu} (\Gamma)$.
\end{example}

We note that Serre duality induces a bijection between $\overline{W}^{\bmu} (C)$ and $\overline{W}^{\bmu^T} (C)$.  Tropically, this corresponds to the fact that the number of maximal chains in $\mathcal{P}_k (\lambda)$ is equal to the number of maximal chains in $\mathcal{P}_k (\lambda^T)$.  If we apply this observation to Example~\ref{Ex:OneCol}, we see that if $\mu_j$ is equal to either $\mu_k$ or $\mu_k - 1$ for each $j > 1$, then
\[
\vert \overline{W}^{\bmu} (C) \vert = \vert \overline{W}^{\bmu} (\Gamma) \vert = 1 .
\]
A similar remark applies to each of the examples below.

\begin{example}
\label{Ex:Catalan}
Let $\bmu = (-3, -2 , \ldots , -2, 0, 0)$.  Then $g = 2k-2$, and $\lambda (\bmu)$ is the partition depicted in Figure~\ref{Fig:Catalan}.

\begin{figure}[H]
\begin{ytableau}
{} & {} & {\square} & {\triangle} \\
{} & {} \\
{} & {} \\
\none[\vdots] & \none[\vdots] \\
{\square} & {\triangle}
\end{ytableau}

\caption{The partition $\lambda (\bmu)$ of Example~\ref{Ex:Catalan}.}
\label{Fig:Catalan}
\end{figure}

If $t$ is a $k$-uniform displacement tableau on $\lambda (\bmu)$, then the restriction of $t$ to the first two columns is a standard Young tableau.  If $t$ has precisely $2k-2$ symbols, then we must have $t(3,1) = t(1,k-1)$ and $t(4,1) = t(2,k-1)$.  (These are the boxes labeled with a square and a triangle, respectively, in Figure~\ref{Fig:Catalan}.)  It follows that $t(2,1) < t(1,k-1)$.  Since the number of standard Young tableaux on the first two columns is the $(k-1)$st Catalan number $C_{k-1}$, and since there is a unique such standard Young tableau $t$ with $t(2,1) > t(1,k-1)$, we see that the number of $k$-uniform displacement tableaux on $\lambda (\bmu)$ with precisely $2k-2$ symbols is $C_{k-1} - 1$.

A general curve $C$ of genus $2k-2$ has gonality $k$, and by Example~\ref{Ex:Classic}, the number of gonality pencils is precisely $C_{k-1}$.  Such a pencil is in $\overline{W}^{\bmu} (C)$ if and only if it is not equal to the distinguished $g^1_k$.  It follows that $\vert \overline{W}^{\bmu} (C) \vert = C_{k-1} - 1$, confirming Conjecture~\ref{Conj:Class} in this case.
\end{example}

\begin{example}
\label{Ex:Trigonal}
If $k=2$, then every splitting type $\bmu$ satisfies the hypotheses of Example~\ref{Ex:OneCol}.  The first interesting examples, therefore, occur when $k$ is equal to 3.  Let $k=3$, and suppose that $\bmu$ is not of the type considered in Example~\ref{Ex:OneCol}.  In other words, $\mu_3 > \mu_2 + 1$, and $\mu_2 > \mu_1 + 1$.  Then $g = 2(\mu_3 - \mu_1) - 3$ is odd, and up to cyclic permutation, we have 
\[
\bC (\lambda (\bmu)) = ( 2\mu_3 - \mu_2 - \mu_1 - 2, \mu_2 - \mu_1 -1, 0) .
\]
We show that
\[
\alpha ( 2\mu_3 - \mu_2 - \mu_1 - 2, \mu_2 - \mu_1 -1, 0) = {{\mu_3 - \mu_1 - 2}\choose{\mu_2 - \mu_1 - 1}} .
\]

One way to see that this formula is invariant under transposition is to note that $\mu_2 - \mu_1 - 1$ is equal to the number of strict rank jumps of size 2, whereas $\mu_3 - \mu_2 - 1$ is equal to the number of strict rank jumps of size 1.  As in Example~\ref{Ex:OneCol}, we prove this by induction.  When $\mu_2 - \mu_1 - 1 = 0$, the result follows from Example~\ref{Ex:OneCol}, and when $\mu_3 - \mu_2 - 1 = 0$, the result follows from the same example applied to $\lambda (\bmu)^T$.  If $z_1 - 1 > z_2 > 0$, then by Lemma~\ref{lem:recur}, we have
\begin{align*}
\alpha ( z_1, z_2, 0) &= \alpha ( 0 , z_2 , z_1 -1) \\
 &= \alpha (z_2 - 1, 0 , z_1 -1 ) + \alpha (0, z_1 - 2, z_2 ) .
\end{align*}
This expression has the following interpretation.  If $\bC (\lambda (\bmu)) = (z_1 , z_2, 0)$, then $\bC (\lambda (\bmu^+)) = (z_2 - 1, 0, z_1 - 1)$ and $\bC (\lambda (\bmu^-)) = (0, z_1 - 2 , z_2)$.  In other words, the number of $k$-saturated tableaux on $\lambda (\bmu)$ is the sum of the number on a partition with one fewer row and the number on a partition with one fewer column.  Evaluating this expression and applying induction, we see that
\[
\alpha ( 2\mu_3 - \mu_2 - \mu_1 - 2, \mu_2 - \mu_1 -1, 0) 
= {{\mu_3 - \mu_1 - 3}\choose{\mu_2 - \mu_1 - 2}} + {{\mu_3 - \mu_1 - 3}\choose{\mu_2 - \mu_1 - 1}} = {{\mu_3 - \mu_1 - 2}\choose{\mu_2 - \mu_1 - 1}} .
\]

In \cite[Theorem~1.1]{LarsonTrigonal}, Larson computes the cardinality of $\overline{W}^{\bmu} (C)$ for a general trigonal curve $C$ of Maroni invariant $n$.  Since $g$ is odd, if $(C,\pi) \in \mathcal{H}_{g,3}$ is general, it has Maroni invariant 1.  Larson's formula then yields the binomial coefficient above, confirming Conjecture~\ref{Conj:Class} for $k=3$.

Example~\ref{Ex:Trigonal} can be generalized to the case where $k$ is arbitrary and
\[
\mu_2 = \mu_3 = \cdots = \mu_{k-1}.
\]
This is done in Example~\ref{Ex:OneRowOneCol} below.
\end{example}

We now consider examples where $k$ is equal to 4, 5, or 6.  We do not consider every splitting type in these cases, considering only the ``maximal'' splitting types in which every strict rank jump has the same size $\alpha$.  If all strict rank jumps of $\bmu$ have size $\alpha$, then all strict rank jumps of $\bmu^T$ have size $k-\alpha$, so it suffices to consider the case where $\alpha \leq \frac{k}{2}$.  Since Example~\ref{Ex:OneCol} is the case where $\alpha = 1$, the first interesting case occurs when $k$ is equal to 4.  We do not know if Conjecture~\ref{Conj:Class} holds for these splitting types, proving it in only a small number of cases.

\begin{example}
\label{Ex:Four}
Let $k=4$, and suppose that $\alpha =2$.  In other words, $\mu_2$ is equal to either $\mu_1$ or $\mu_1 + 1$, and $\mu_3$ is equal to either $\mu_4$ or $\mu_4 - 1$.  In this case we see that, up to cyclic permutation, $\bC (\lambda (\bmu))$ is either of the form $(z,z,0,0)$ or $(z+1,z-1,0,0)$.  We show, by induction on $z$, that
\[
\alpha (z,z,0,0) = \alpha (z+1,z-1,0,0) = 2^{z-1} .
\]

The base case, when $z$ is equal to 1, is covered by Example~\ref{Ex:OneCol}.  For the inductive step, by Lemma~\ref{lem:recur}, we see that
\begin{align*}
\alpha (z,z,0,0) = \alpha (0,z,0,z-1) &= \alpha (z-1,0,0,z-1) + \alpha (0,z,z-2,0) \\
&= 2^{z-2} + 2^{z-2} = 2^{z-1} \\
\alpha (z+1,z-1,0,0) = \alpha (0,z-1,0,z) & = \alpha (z-2,0,0,z) + \alpha (0,z-1,z-1,0) \\
&= 2^{z-2} + 2^{z-2} = 2^{z-1} .
\end{align*}
As in Example~\ref{Ex:Trigonal}, the expressions on the right are equal to $\alpha (\bC (\lambda (\bmu^+))) + \alpha (\bC (\lambda (\bmu^-)))$.

In general, we do not know if Conjecture~\ref{Conj:Class} holds in this case.  It holds for $z \leq 2$ by Example~\ref{Ex:Classic}, and for $z = 3$ by Example~\ref{Ex:Catalan}.  We will see in Example~\ref{Ex:Quadric} below that it also holds for the splitting type $\bmu = (-3,-3,0,0)$, in which case $z=4$.
\end{example}

\begin{example}
\label{Ex:Five}
Let $k=5$, and suppose that $\alpha = 2$.  In other words, $\mu_2$ and $\mu_3$ are equal to either $\mu_1$ or $\mu_1 + 1$, and $\mu_4$ is equal to either $\mu_5$ or $\mu_5 - 1$.  Up to cyclic permutation, $\bC (\lambda (\bmu))$ is either of the form $(z,z,0,0,0)$ or $(z+2,z-1,0,0,0)$.  We show, by induction on $z$, that
\begin{align*}
\alpha (z,z,0,0,0) &= F_{2z-2} \\
\alpha (z+2,z-1,0,0,0) &= F_{2z-1},
\end{align*}
where $F_n$ denotes the $n$th Fibonacci number.  The base case, where $z=1$, follows from Example~\ref{Ex:OneCol}.  For the inductive step, by Lemma~\ref{lem:recur}, we have
\begin{align*}
\alpha (z,z,0,0,0) &= \alpha (0,z,0,0,z-1) \mbox{ and} \\
\alpha (z+2,z-1,0,0,0) &= \alpha (0,z-1,0,0,z+1) ,
\end{align*}
so we will also show by induction that $\alpha (0,z,0,0,z-1) = F_{2z-2}$ and $\alpha (0,z-1,0,0,z+1) = F_{2z-1}$.  Again, the base cases follow from Example~\ref{Ex:OneCol}.  Together with the inductive hypothesis, by Lemma~\ref{lem:recur}, we have
\begin{align*}
\alpha (0,z,0,0,z-1) &= \alpha (z-1,0,0,0,z-1) + \alpha (0,z,0,z-2,0) \\
&= F_{2z-4} + F_{2z-3} = F_{2z-2} \\
\alpha (0,z-1,0,0,z+1) &= \alpha (z-2,0,0,0,z+1) + \alpha (0,z-1,0,z,0) \\
&= F_{2z-3} + F_{2z-2} = F_{2z-1} .
\end{align*}

Conjecture~\ref{Conj:Class} holds when $-2 \leq \mu_j \leq 0$ for all $j$ by Example~\ref{Ex:Classic}, and when $\bmu = (-3,-2,-2,0,0)$ by Example~\ref{Ex:Catalan}.  We will see in Example~\ref{Ex:Quadric} below that it also holds when $\bmu = (-3,-3,-2,0,0)$.  We now show that it holds when $\bmu = (-3,-3,-2,-1,0)$.

In this case, $g=7$, $\bC (\lambda (\bmu)) = (5,2,0,0,0)$, and $\alpha (5,2,0,0,0) = F_5 = 8$.  For $(C,\pi) \in \mathcal{H}_{7,5}$, we see that $D \in \overline{W}^{\bmu} (C)$ if and only if $D$ is effective of degree 2 and $K_C - g^1_5 - D$ has rank at least 1.  By Riemann-Roch, the divisor class $K_C - g^1_5$ has degree 7 and rank 2.  The image of $C$ under the complete linear series $\vert K_C - g^1_5 \vert$ is a plane curve of degree 7, with ${{7-1}\choose{2}} - 7 = 8$ nodes.  An effective divisor $D$ satisfies $\rk (K_C - g^1_5 - D) \geq 1$ if and only if the image of $D$ under this map is a single point.  It follows that the divisor classes in $\overline{W}^{\bmu} (C)$ are precisely the preimages of the nodes, and thus that $\vert \overline{W}^{\bmu} (C) \vert = 8$.  
\end{example}

\begin{example}
\label{Ex:Six2}
Let $k=6$, and suppose that $\alpha = 2$.  Up to cyclic permutation, $\bC (\lambda (\bmu))$ is either of the form $(z,z,0,0,0,0)$ or $(z+2,z-2,0,0,0,0)$.  We show, by induction on $z$, the following formulas:
\begin{align*}
\alpha (z,z,0,0,0,0) = \alpha (0,z,0,0,0,z-1) &= \frac{3^{z-1} + 1}{2} \\
\alpha (z+2,z-2,0,0,0,0) = \alpha (0,z-2,0,0,0,z+1) &= \frac{3^{z-1}-1}{2} \\
\alpha (z+1,0,0,z-1,0,0) &= 3^{z-1} .
\end{align*}

The base cases, when $z=1$ on the first and third line, or when $z=2$ on the second line, follow from Example~\ref{Ex:OneCol}.  The first equality on each of the first two lines above follows directly from Lemma~\ref{lem:recur}.  For the inductive step, by induction together with Lemma~\ref{lem:recur}, we have
\begin{align*}
\alpha (0,z,0,0,0,z-1) &= \alpha (z-1,0,0,0,0,z-1) + \alpha (0,z,0,0,z-2,0) \\ 
&= \frac{3^{z-2} + 1}{2} + 3^{z-2} = \frac{3^{z-1} + 1}{2} \\
\alpha (0,z-2,0,0,0,z+1) &= \alpha (z-3,0,0,0,0,z+1) + \alpha (0,z-2,0,0,z,0) \\
&= \frac{3^{z-2} - 1}{2} + 3^{z-2} = \frac{3^{z-1} - 1}{2} \\
\alpha (z+1,0,0,z-1,0,0) &= \alpha (0,0,0,z-1,0,z) + \alpha (z+1,0,z-2,0,0,0) \\
&= \frac{3^{z-1} + 1}{2} + \frac{3^{z-1} - 1}{2} = 3^{z-1} .
\end{align*}

Conjecture~\ref{Conj:Class} holds when $z \leq 3$, and for the splitting type $\bmu = (-2,-2,-2,-2,0,0)$ by Example~\ref{Ex:Classic}. It also holds for the splitting type $\bmu = (-3,-2,-2,-2,0,0)$ by Example~\ref{Ex:Catalan}.  The splitting types $\bmu = (-3,-3,-2,-2,0,0)$ and $\bmu^T = (-3,-3,-1,-1,0,0)$ will make an appearance in Example~\ref{Ex:Quadric} below. 

\end{example}

\begin{example}
\label{Ex:Six3}
Let $k=6$, and suppose that $\alpha = 3$.  Up to cyclic permutation, $\bC (\lambda (\bmu))$ is of the form $(z,z,z,0,0,0)$, $(z+1,z+1,z-2,0,0,0)$, or $(z+2,z-1,z-1,0,0,0)$.  To formulate expressions in these cases, we first introduce the function
\begin{displaymath}
\beta (z) := \left\{ \begin{array}{ll}
2 & \textrm{if $z \equiv 0 \pmod{3}$} \\
-1 & \textrm{otherwise.}
\end{array}\right.
\end{displaymath}
Note that $\beta (z-1) + \beta (z) = - \beta (z+1)$.
By a similar argument to Examples~\ref{Ex:Four}, ~\ref{Ex:Five}, and~\ref{Ex:Six2}, we obtain the following formulas.
\begin{align*}
\alpha (z,z,z,0,0,0) = \alpha (0,z,z,0,0,z-1) &= \frac{2^{3z-2} + (-1)^z\beta (z)}{3} \\
\alpha (z+1,z+1,z-2,0,0,0) = \alpha (0,z+1,z-2,0,0,z) &= \frac{2^{3z-2} + (-1)^z \beta (z-1)}{3} \\
\alpha (z+2,z-1,z-1,0,0,0) = \alpha (0,z-1,z-1,0,0,z+1) &= \frac{2^{3z-2} + (-1)^z \beta (z+1)}{3} \\
\alpha(z-1,0,z,0,z+1,0) &= 2^{3z-2} .
\end{align*}

We will consider the splitting type $\bmu = (-3,-3,-2,-1,0,0)$ in Example~\ref{Ex:Quadric} below.  The Hasse diagram pictured in Figure~\ref{Fig:Hasse} is that of $\mathcal{P}_6 (\lambda (\bmu))$.
\end{example}

\begin{example}
\label{Ex:Quadric}
Let $(C,\pi) \in \mathcal{H}_{2k,k}$ be general, and let $L = K_C - g^1_k$.  By Riemann-Roch, $h^0 (C,L) = k+1$, and we consider the image of $C$ in $\PP^k$ under the complete linear series $\vert L \vert$.  We have
\begin{align*}
\mathrm{exp dim} H^0 (\PP^k , \mathcal{I}_C (2)) &= \dim \Sym^2 H^0 (C,L) - \dim H^0 (C,2L) \\
& = {{k+2}\choose{2}} - (4k-3).
\end{align*}
The variety $X_4$ parameterizing quadrics of rank at most 4 in $\PP^k$ has dimension $4k-2$, so one expects the curve $C$ to be contained in a finite number of rank 4 quadrics.  The expected number of rank 4 quadrics in $H^0 (\PP^k, \mathcal{I}_C (2))$ is
\[
\deg X_4 = \frac{{{k+1}\choose{k-3}}{{k+2}\choose{k-4}} \cdots {{2k-3}\choose{1}}}{{{1}\choose{0}}{{3}\choose{1}}{{5}\choose{2}} \cdots {{2k-7}\choose{k-4}}} \mbox{ \cite{HarrisTu}}.
\]
Each rank 4 quadric is a cone over $\PP^1 \times \PP^1$, and the pullback of $\cO (1)$ from each of the two factors yields a pair of line bundles on $C$, each of rank 1, whose tensor product is $L$.

Conversely, given a pair of divisor classes $D, D'$, each of rank 1, such that $D+D' = L$, we obtain a rank 4 quadric in $\PP^k$ containing $C$.  To see this, let $s_0 , s_1$ be a basis for $H^0 (C,D)$ and $t_0 , t_1$ be a basis for $H^0 (C,D')$.  Then the entries of the $2 \times 2$ matrix $M_{ij} = (s_i \otimes t_j)$ are linear forms in $\PP^k$, and the determinant of this matrix is a rank 4 quadric that vanishes on $C$.  In other words, each rank 4 quadric corresponds to a pair of divisors in the set
\begin{align*}
\left\{ D \in \Pic (C) \mbox{ } \vert \mbox{ } h^0(C,D) = h^0 (C,L-D) = 2  \right\} & \\ 
= \Big( \bigcup_{i=0}^{k-4} W^{(-3^{(2)},-2^{(i)},-1^{(k-4-i)},0^{(2)})} (C) \Big) \cup \{ g^1_k \} \cup \{ L - g^1_k \} .
\end{align*}

Since $(C,\pi)$ is general, the splitting type loci in the union above are all smooth of dimension zero, and we see that 
\[
2 + \sum_{i=0}^{k-4} \Big\vert W^{(-3^{(2)},-2^{(i)},-1^{(k-4-i)},0^{(2)})} (C) \Big\vert = 2 \frac{{{k+1}\choose{k-3}}{{k+2}\choose{k-4}} \cdots {{2k-3}\choose{1}}}{{{1}\choose{0}}{{3}\choose{1}}{{5}\choose{2}} \cdots {{2k-7}\choose{k-4}}} .
\]
We now show that this expression holds for $\Gamma$ when $k \leq 6$.  By Example~\ref{Ex:Four}, when $k=4$, we have
\[
2+ \Big\vert W^{(-3,-3,0,0)} (\Gamma) \Big\vert = 2 + 2^3 = 10 = 2 {{5}\choose{1}} .
\]
By Example~\ref{Ex:Five}, when $k=5$, we have
\begin{align*}
2 + \Big\vert W^{(-3,-3,-2,0,0)} (\Gamma) \Big\vert + \Big\vert W^{(-3,-3,-1,0,0)} (\Gamma) \Big\vert \\
= 2 + F_8 + F_8 = 2+ 34 + 34 = 70 = 2 \frac{{{6}\choose{2}}{{7}\choose{1}}}{{{1}\choose{0}}{{3}\choose{1}}} .
\end{align*}
By Examples~\ref{Ex:Six2} and~\ref{Ex:Six3}, when $k=6$, we have
\begin{align*}
2+ \Big\vert W^{(-3,-3,-2,-2,0,0)} (\Gamma) \Big\vert + \Big\vert W^{(-3,-3,-2,-1,0,0)} (\Gamma) \Big\vert + \Big\vert W^{(-3,-3,-1,-1,0,0)} (\Gamma) \Big\vert \\
= 2 + \frac{3^5 + 1}{2} + \frac{2^{10} + 2}{3} + \frac{3^5 + 1}{2} = 2+ 122 + 342 + 122 = 588 = 2 \frac{{{7}\choose{3}}{{8}\choose{2}}{{9}\choose{1}}}{{{1}\choose{0}}{{3}\choose{1}}{{5}\choose{3}}} .
\end{align*}
\end{example}

\begin{example}
\label{Ex:OneRowOneCol}
We now consider the case where $k$ is arbitrary and
\[
\mu_2 = \mu_3 = \cdots = \mu_{k-1}.
\]
The cases where $\mu_k \leq \mu_{k-1} +1$ or $\mu_1 \geq \mu_2 - 1$ are covered in Example~\ref{Ex:OneCol}, so we assume otherwise.  For ease of notation, we write $z_1 = (k-1)(\mu_k -1) - (k-2)\mu_2 - \mu_1$ and $z_2 = \mu_2 - \mu_1 - 1$.  Then $\bC (\lambda(\bmu)) = ( z_1, z_2^{(k-2)} , 0 )$, and we will show in Lemma~\ref{lem:OneRowOneCol} below that 
\[
\alpha ( z_1, z_2^{(k-2)} , 0 ) = {{(k-2)(\mu_k - \mu_1 - 2)}\choose{(k-2)(\mu_2 - \mu_1 - 1)}} .
\]

This expression matches the cardinality of $\overline{W}^{\bmu} (C)$ for general $(C,\pi) \in \mathcal{H}_{g,k}$.  To see this, following \cite[Lemma~2.2]{LarsonTrigonal}, we see that
\[
W^{\bmu} (C) = \left\{ D \in \Pic^{d(\bmu)} (C) \mbox{ } \vert \mbox{ } h^0 (D - \mu_k g^1_k) = h^0 (K_C - D + (\mu_1 + 2)g^1_k ) = 1  \right\} .
\]
In other words, $D \in W^{\bmu} (C)$ if and only if $D = \mu_k g^1_k + E$, where $E$ is an effective divisor of degree $(k-2)(\mu_2 - \mu_1 - 1)$, such that $K_C - (\mu_k - \mu_1 - 2)g^1_k - E$ is also effective.  Note that 
\[
\deg \Big( K_C - (\mu_k - \mu_1 -2)g^1_k \Big) = (k-2) (\mu_k - \mu_1 - 2) .
\]
Since $C$ is general, $K_C - (\mu_k - \mu_1 -2)g^1_k$ is equivalent to a unique effective divisor.  If this divisor is a sum of distinct points, then the set of divisor classes $E$ satisfying the conditions above is simply the set of subsets of these points of size $(k-2)(\mu_2 - \mu_1 - 1)$.  We therefore see that $\vert W^{\bmu} (C) \vert$ is equal to the binomial coefficient above.

\begin{lemma}
\label{lem:OneRowOneCol}
Let $z_1 \geq z_2 \geq 0$ be integers, let $\vec{z}_i (z_2) = (z_2^{(k-2-i)}, (z_2 -1)^{(i)}, 0)$, and let $\vec{z}_{ij} (z_1 , z_2)$ be the vector obtained from $\vec{z}_i (z_2)$ by inserting $z_1$ between entries $j$ and $j+1$.  Then
\[
\alpha (\vec{z}_{ij} (z_1 , z_2 )) = {{\Big\lfloor \frac{k-2}{k-1} \Big( z_1 + (k-2)z_2 \Big) \Big\rfloor - i}\choose{(k-2)z_2 - i}} .
\]
\end{lemma}

\begin{proof}
Note that the expression $(z_1 + (k-2)z_2)$ is divisible by $k-1$ if and only if $z_1 \equiv z_2 \pmod{k-1}$.  If $\bC (\lambda) = \vec{z}_{ij} (z_1 , z_2)$, then this congruence holds if and only if the partition $\lambda'$, obtained by deleting all columns of $\lambda$ that are taller than $z_2$, has an outside corner in $D_{z_1}$.  Since $\bC (\lambda') = \vec{z}_{ij} (z_2, z_2)$, this holds if and only if $j = k-2$.

We establish the above formula by induction.  The base cases, where $z_1 = z_2$, or $z_2 = i = 0$, both follow from Example~\ref{Ex:OneCol}.  If $j = k-1$, then by Lemma~\ref{lem:recur}, we have
\[
\alpha (\vec{z}_{i(k-1)} (z_1 , z_2 )) = \alpha (\vec{z}_{i(k-2)} (z_1 - 1, z_2 ) )  .
\]
By induction, the expression on the right is equal to
\[
{{\Big\lfloor \frac{k-2}{k-1} \Big( z_1 + (k-2)z_2 - 1 \Big) \Big\rfloor - i}\choose{(k-2)z_2 - i}} .
\]
Since $j = k-1$, by the above we see that $z_1 \equiv z_2 + 1 \pmod{k-1}$, so the term $(z_1 + (k-2)z_2 -1)$ is divisible by $k-1$.  The expression above is therefore equal to
\[
{{\Big\lfloor \frac{k-2}{k-1} \Big( z_1 + (k-2)z_2 \Big) \Big\rfloor - i}\choose{(k-2)z_2 - i}} .
\]

Otherwise, if $j < k-1$, then by Lemma~\ref{lem:recur}, we have
\[
\alpha (\vec{z}_{ij} (z_1 , z_2 ) ) = \alpha (\vec{z}_{i(j-1)} (z_1 - 1, z_2)) + \alpha (\vec{z}_{(i-1)j} (z_1 , z_2)).
\]
By induction, the expression on the right is equal to
\begin{align*}
& {{\Big\lfloor \frac{k-2}{k-1} \Big( z_1 + (k-2)z_2 \Big) \Big\rfloor - (i+1)}\choose{(k-2)z_2 - (i+1)}} + {{\Big\lfloor \frac{k-2}{k-1} \Big( z_1 + (k-2)z_2 - 1 \Big) \Big\rfloor - i}\choose{(k-2)z_2 - i}} \\
= & {{\Big\lfloor \frac{k-2}{k-1} \Big( z_1 + (k-2)z_2 \Big) \Big\rfloor - (i+1)}\choose{(k-2)z_2 - (i+1)}} + {{\Big\lfloor \frac{k-2}{k-1} \Big( z_1 + (k-2)z_2 \Big) \Big\rfloor - (i+1)}\choose{(k-2)z_2 - i}} \\
= & {{\Big\lfloor \frac{k-2}{k-1} \Big( z_1 + (k-2)z_2 \Big) \Big\rfloor - i}\choose{(k-2)z_2 - i}} ,
\end{align*}
where the second line holds because $(z_1 + (k-2)z_2 - 1)$ is not divisible by $k-1$.
\end{proof}

\end{example}

\bibliography{references}{}
\bibliographystyle{alpha}
\end{document}